\documentclass[12pt]{amsart}

\usepackage{amsmath}
\usepackage[backref=page]{hyperref}
\usepackage{tikz,float,caption}
\usetikzlibrary{arrows.meta,calc,decorations.markings,patterns,cd,patterns.meta}

\usepackage[margin=1.4 in,top=1.2in, bottom=1.4 in]{geometry}

\usepackage[bb=ams,scr=euler]{mathalpha}

\usepackage{setspace}
\usepackage{enumitem}
\setenumerate{
  wide=10pt,
  leftmargin=0pt,
  labelwidth=*,
  label=(\roman*),
  topsep=0pt,
  listparindent=0pt,
  parsep=4pt}

\usepackage{thmtools}
\declaretheoremstyle[headfont=\normalsize\normalfont\bfseries,notefont=\mdseries, notebraces={(}{)},bodyfont=\normalfont,postheadspace=0.5em]{basicstyle}
\declaretheoremstyle[headfont=\normalsize\normalfont\bfseries,notefont=\mdseries,
notebraces={(}{)},bodyfont=\normalfont\itshape,postheadspace=0.5em]{italstyle}

\declaretheorem[style=italstyle,name=Theorem,numberwithin=section]{theorem}
\declaretheorem[style=italstyle,name=Corollary,sibling=theorem]{cor}
\declaretheorem[style=italstyle,name=Conjecture,sibling=theorem]{conjecture}

\declaretheorem[style=italstyle,name=Proposition,sibling=theorem]{prop}
\declaretheorem[style=italstyle,name=Lemma,sibling=theorem]{lemma}
\declaretheorem[style=basicstyle,name=Proof,numbered=no,qed=$\square$]{preproof}
\renewenvironment{proof}{\preproof}{\endpreproof}

\makeatletter
\newcommand{\bd}{\partial}

\newcommand{\C}{\mathbb{C}}
\renewcommand{\d}{\mathrm{d}}
\newcommand{\Ham}{\mathrm{Ham}}
\newcommand{\id}{\mathrm{id}}

\newcommand{\intprod}{\mathbin{{\tikz{\draw(-0.1,0)--(0.1,0)--(0.1,0.2)}\hspace{0.5mm}}}}

\newcommand{\R}{\mathbb{R}}
\renewcommand\section{\@startsection{section}{1}{0pt}{-3.5ex \@plus -1ex \@minus -.2ex}{2.3ex \@plus.2ex}{\centering\itshape}}
\newcommand{\set}[1]{\left\{#1\right\}}
\renewcommand{\subsection}{\@startsection{subsection}{2}%
  \z@{.5\linespacing\@plus.7\linespacing}{-.5em}%
  {\normalfont\itshape}}
\newcommand{\Z}{\mathbb{Z}}

\makeatother

\title[Extensible positive loops]{Extensible positive loops and vanishing of symplectic cohomology}
\author{Dylan Cant}
\author{Jakob Hedicke}
\author{Eric Kilgore}
\begin{document}

\begin{abstract}
  The symplectic cohomology of certain symplectic manifolds $W$ with non-compact ends modelled on the positive symplectization of a compact contact manifold $Y$ is shown to vanish whenever there is a positive loop of contactomorphisms of $Y$ which extends to a loop of Hamiltonian diffeomorphisms of $W$. An open string version of this result is also proved: the wrapped Floer cohomology of a Lagrangian $L$ with ideal Legendrian boundary $\Lambda$ is shown to vanish if there is a positive loop $\Lambda_{t}$ based at $\Lambda$ which extends to an exact loop of Lagrangians based at $L$. Various examples of such loops are considered. Applications include the construction of exotic compactly supported symplectomorphisms and exotic fillings of $\Lambda$.
\end{abstract}

\maketitle

%%% \S1
\section{Introduction}
\label{sec:introduction}

\subsection{Statement of results}
\label{sec:statement-results-1}

Let $(W,\omega)$ be a convex-at-infinity symplectic manifold, and let $(Y,\xi)$ be its ideal contact boundary. Recall that this means the non-compact end of $W$ is symplectomorphic to the positive half of the symplectization of $Y$. In particular, $W$ has a Liouville form $\lambda$ outside of a compact set. The associated Liouville vector field $Z$ (extended arbitrarily to the compact part of $W$) is complete. As is well-known, any symplectomorphism which is equivariant with respect to $Z$, outside of a compact set, has an \emph{ideal restriction} to a contactomorphism of the ideal boundary $Y$.

A symplectic isotopy $\psi_{t}$ is called a \emph{contact-at-infinity Hamiltonian system} provided $\varphi_{0}=1$, its infinitesimal generator $X_{t}$ is $1$-periodic and $\omega$-dual to an exact 1-form, and $\varphi_{t}$ commutes with the Liouville flow outside of a compact set. Each such system has an ideal restriction to a contact isotopy of $Y$.

A loop of contactomorphisms $\varphi_{t}$ of $Y$ based at the identity is said to be \emph{extensible} if it is the ideal restriction of a contact-at-infinity Hamiltonian loop. As explained in \S\ref{sec:fibr-sequ-sympl}, contractible loops are always extensible.

A loop of contactomorphisms is called \emph{positive} provided the curves $\varphi_{t}(y)$ are positively transverse to the contact distribution in $Y$ (with respect to the coorientation induced by the Liouville form on the end of $W$).

Our first main theorem asserts that the existence of an extensible positive loop ensures the vanishing of symplectic cohomology.

\subsubsection{Symplectic cohomology and extensible positive loops}
\label{sec:sympl-cohom-spec}

In order to define Floer theoretic invariants, we assume that $W$ is \emph{symplectically atoroidal}, in that the integral of $\omega$ over any smooth $2$-torus is zero. The prototypical example one should have in mind is that of a Liouville manifold, when the Liouville form extends smoothly to all of $W$.

The atoroidal assumption allows us to define the symplectic cohomology $\mathrm{SH}(W)$ invariant; our convention is to define this as the colimit of Hamiltonian Floer cohomology groups for contact-at-infinity systems whose ideal restrictions become more and more positive; see \S\ref{sec:floer-cohom-groups}. This invariant is well-known, see, e.g.,  \cite{floer_hofer_sh_i,cieliebak_floer_hofer_sh_ii,viterbo_functors_and_computations_1,seidel-biased,ritter_TQFT,ritter_negative_line_bundles,ritter_circle_actions,uljarevic_floer_homology_domains,merry_ulja,uljarevic_ssh,ulja_zhang,shelukhin_viterbo,shelukhin_zoll,PA_spectral_diameter}. In this paper, Floer cohomology is defined over the field $\Z/2$, and considers all orbits (not just the contractible orbits).

\begin{theorem}\label{theorem:main-absolute}
  If the ideal boundary of a convex-at-infinity and symplectically atoroidal manifold $W$ admits an extensible positive loop of contactomorphisms then the symplectic cohomology of $W$ vanishes.
\end{theorem}

The method used to prove Theorem \ref{theorem:main-absolute} involves naturality isomorphisms relating the Floer cohomology of a system $\psi_{t}$ with the Floer cohomology of $\varphi_{t}\circ \psi_{t}$ where $\varphi_{t}$ is a loop of Hamiltonian diffeomorphisms. Our argument is directly inspired by work of \cite{merry_ulja,uljarevic_floer_homology_domains}, which prove the symplectic cohomology satisfies certain dimension bounds in the presence of extensible positive loops, and the work of \cite{ritter_negative_line_bundles,ritter_circle_actions} which investigates the symplectic cohomology of certain manifolds whose ideal boundaries admit periodic Reeb flows. If one restricts to contractible positive loops of contactomorphisms, then our vanishing result appears already in \cite{chantraine_colin_d_rizell} for Liouville manifolds via a different method.

A similar vanishing result for Rabinowitz Floer homology (RFH) in the presence of contractible positive loops appears in \cite{albers_merry_orderability_non_squeezing_RFH}, using the method of spectral invariants; see also \cite{djordjevic_uljarevic_zhang,cant_sh_barcode} which prove the contractible case using spectral invariants built from symplectic cohomology groups.\footnote{One should note that \cite{albers_merry_orderability_non_squeezing_RFH,djordjevic_uljarevic_zhang,cant_sh_barcode} suppose $W$ is a Liouville manifold} See \cite[Theorem 6.4]{kwon-van-koert} for a related (weaker) statement to our Theorem \ref{theorem:main-absolute}. In certain cases, our result follows from \cite{ritter_negative_line_bundles,ritter_circle_actions} which considers the effect of Hamiltonian circle actions (with positive ideal restriction) on the symplectic and quantum cohomologies of certain convex-at-infinity manifolds; see also \cite{venkatesh_quantitative_nature}. It is worth mentioning that the ideal restrictions considered in \cite{ritter_negative_line_bundles,ritter_circle_actions,venkatesh_quantitative_nature} are required to be strict contactomorphisms for some contact form.

See also \cite{ritter-zivanovic-1,ritter-zivanovic-2} for analysis of the symplectic cohomology of a large class of open symplectic manifolds $W$ admitting special kinds of Hamiltonian circle actions, including examples of $W$ which are \emph{not} convex-at-infinity.

\subsubsection{The open string analogue}
\label{sec:open-string-analogue}

Given a Lagrangian $L\subset W$ which is tangent to the Liouville vector field at infinity, one can associate its ideal Legendrian boundary $\Lambda\subset Y$. In this case, we say that $L$ is \emph{contact-at-infinity}, and call $\Lambda$ its \emph{ideal restriction}.

A path of contact-at-infinity Lagrangians $L_{t}$ is called \emph{exact} provided it is induced by a contact-at-infinity Hamiltonian path $L_{t}=\varphi_{t}(L)$. Every exact path has an ideal restriction to a Legendrian isotopy. A loop of Legendrians based at $\Lambda$ is called \emph{extensible} provided it is the ideal restriction of an exact loop based at $L$; see \S\ref{sec:lagrangian-ideal-restriction} for the precise definition.

The open-string invariant corresponding to the symplectic cohomology is the \emph{wrapped Floer cohomology} $\mathrm{HW}(L)$; see, e.g., \cite{abouzaid_seidel_open_string_analogue,ritter_TQFT,ganatra_pardon_shende}. In this case, we require that $L$ is \emph{symplectically acylindrical}, i.e., every map of a cylinder $[0,1]\times \R/\Z\to (W,L)$ has zero symplectic area.

Our convention is that $\mathrm{HW}(L)$ is a colimit of Floer cohomology groups $\mathrm{HF}(L;\psi_{t})$ where $\psi_{t}$ is a contact-at-infinity Hamiltonian system whose ideal restriction becomes more and more positive. We show:
\begin{theorem}\label{theorem:main-relative}
  If there is an extensible positive loop of Legendrians based at the ideal boundary of an acylindrical Lagrangian $L$, then the wrapped Floer cohomology of $L$ vanishes.
\end{theorem}

The paper of \cite{chantraine_colin_d_rizell} proves this in the case of contractible positive loops, when $L$ is exact and $W$ is a Liouville manifold. As in the closed-string case, our proof is based on the naturality transformation trick introduced in \cite{uljarevic_floer_homology_domains,merry_ulja,uljarevic_ssh} and \cite{ritter_negative_line_bundles,ritter_circle_actions}.

\subsection{Discussion of results}
\label{sec:discussion-results}

\subsubsection{Existence of Reeb chords}
\label{sec:exist-reeb-chords}

It is well-known, using ideas of \cite{viterbo_functors_and_computations_1}, that the vanishing of symplectic cohomology implies the ideal boundary $Y$ has a closed Reeb orbit for any choice of contact form. Thus Theorem \ref{theorem:main-absolute} implies $Y$ has a closed Reeb orbit whenever $Y$ admits an extensible positive loop for some atoroidal filling $W$. This is of course less general than \cite{albers_fuchs_merry_orderability_weinstein} which proves that a closed contact manifold admitting \emph{any} positive loop of contactomorphisms has a closed Reeb orbit.

However, the open string analogue of this argument also holds; see \cite{abouzaid_seidel_open_string_analogue,ritter_TQFT}. Therefore Theorem \ref{theorem:main-relative} implies:

\begin{theorem}\label{theorem:existence-reeb-chords}
  A Legendrian $\Lambda$ appearing as the ideal boundary of an acylindrical and contact-at-infinity Lagrangian $L$ admits a Reeb chord whenever there is a positive loop based at $\Lambda$ which is the ideal restriction of an exact loop of Lagrangians based at $L$.
\end{theorem}

This is an immediate consequence of Theorem \ref{theorem:main-relative}. Comparing with \cite{albers_fuchs_merry_orderability_weinstein}, one naturally wonders whether every Legendrian which is the basepoint of a positive loop (not assumed to be extensible) has a Reeb chord for every choice of contact form.

\subsubsection{1-stabilizations and extensible positive loops}
\label{sec:1-stab-spec}

Restrict to the case when $W$ is a Liouville manifold, and fix a Liouville form $\lambda$. Define the $1$\emph{-stabilization} of $W$ to be $W'=W\times \mathbb{C}$, equipped with the Liouville form $\lambda'=\lambda+\frac{1}{2}(x\d y-y\d x)$. We show that:
\begin{prop}\label{prop:extensible-positive}
  The $1$-stabilization of a Liouville manifold admits an extensible positive loop.
\end{prop}
This implies that the $n$-stabilization $W\times \mathbb{C}^n$ admits an extensible positive loop for any $n\geq 1$. In the case $n\ge 2$, the existence of contractible positive loops is shown in \cite[Section 3]{ekp}. The loops we construct are generally not contractible, even in the case $n\ge 2$.

It is perhaps interesting to note that general convex-at-infinity manifolds cannot be stabilized (i.e., if $W$ is convex-at-infinity, then $W\times \C$ may no longer be convex-at-infinity). For instance, if $W$ contains a compact symplectic manifold, then $W\times \C$ can never be exact outside of a compact set.

In \S\ref{sec:stab-liouv-manif} we show that every stabilization $W'$ admits an extensible \emph{non-negative} loop (simply by rotating the $\C$ factor). Then the ergodic trick of \cite{ep2000} upgrades this to imply $W'$ admits an extensible positive loop; see \S\ref{sec:ergod-trick-eliashb}. This completes the proof of Proposition \ref{prop:extensible-positive}.

In particular, we recover the result of \cite[Proposition 4.5]{viterbo_functors_and_computations_1} and \cite[Proposition 2]{oancea-kunneth-formula-sh} that 1-stabilizations have vanishing symplectic cohomology.

If $L\subset W$ is a Lagrangian in a Liouville manifold $W$, say that $L\times \R\subset W\times \C$ is the \emph{stabilization} of $L$. The above discussion implies any stabilized Lagrangian is the basepoint of an extensible positive loop.

It is also noteworthy that the ergodic arguments of \cite{ep2000} can be modified to show that the existence of a non-constant \emph{non-negative} extensible loop of Legendrians implies the existence of an extensible positive loop; see \cite[\S4.2]{chernov_nemirovski_universal} for further discussion.

\subsubsection{A partial order on a cover of the contactomorphism group}
\label{sec:extensible-cover}

In \cite{ep2000} Eliashberg and Polterovich introduced a relation on the universal cover of the group of contactomorphisms.
This relation is a partial order if and only if there are no contractible positive loops of contactomorphisms.

It is a natural question to ask if the non-existence of extensible positive loops gives rise to a partial order on some smaller cover of $\mathrm{Cont}(Y)$.
In fact such a cover can be constructed as follows.

Let $G:=\pi_1(\mathrm{Cont}(Y))$ and define: $$H:=\{[\psi_t]\in G: \psi_t \text{ is extensible relative $W$}\}.$$
One can easily check that $H$ is a normal subgroup of $G$;
hence there exists a normal covering $\pi\colon\Gamma(W)\rightarrow \mathrm{Cont}(Y)$
with fundamental group being $H$ and group of deck transformations being $G/H$.

By construction $\Gamma(W)$ can be identified with the set of equivalence classes of paths in $\mathrm{Cont}(Y)$ starting at $\id$, where $\phi_t\sim\phi_t'$ if and only if there exists an extensible loop $\psi_t$ with $\phi_t'=\phi_t\circ\psi_t$.
The cover $\Gamma(W)$ has a well-defined group structure by setting $[\phi_t][\phi_t']:=[\phi_t\phi_t']$, with the identity element being the class of extensible loops.
Hence a bi-invariant relation can be defined on $\Gamma(W)$ by setting $\id\leq [\phi_t]$ if and only if $[\phi_t]$ is represented by a non-negative path of contactomorphisms.

As in \cite{ep2000}, the relation on $\Gamma(W)$ is a partial order if and only if there exist no extensible positive loops of contactomorphisms; the proof uses Lemma \ref{lem:ergodic}. Theorem \ref{theorem:main-absolute} implies that this is the case whenever $\mathrm{SH}(W)\neq 0$.

\subsubsection{Exotic symplectomorphisms associated to inextensible loops}
\label{sec:exot-sympl}

The authors learned the relationship between inextensible loops and exotic symplectomorphisms from the work of Uljarevi\'c, in particular \cite{uljarevic_floer_homology_domains,merry_ulja,ulja_drobnjak}.

Let $\mathrm{Ham}(W)$ be the space of time-1 maps of contact-at-infinity Hamiltonian systems; as usual, this forms a group. The topologies considered on $\mathrm{Ham}(W)$ and $\mathrm{Cont}(Y)$ are described in \S\ref{sec:topol-group-cont}.

In \S\ref{sec:fibr-sequ-sympl} we recall the arguments proving $\mathrm{Ham}(W)\to \mathrm{Cont}(Y)$ is a Serre fibration. The fiber over the identity is the group $\mathrm{Ham}_{c}(W):=\mathrm{Ham}(W)\cap \mathrm{Diff}_{c}(W)$. The fibration property induces a connecting morphism:
\begin{equation}\label{eq:connecting-homo}
  \pi_{1}(\mathrm{Cont}(Y),1)\to \pi_{0}(\mathrm{Ham}_{c}(W)).
\end{equation}
It is easy to see that an inextensible loop (based at the identity) hits a non-trivial element in $\pi_{0}(\mathrm{Ham}_{c}(W))$.
We call such a mapping class ``exotic'' in that it lives in $\mathrm{Ham}_{c}(W)=\mathrm{Ham}(W)\cap \mathrm{Diff}_{c}(W)$, but cannot be generated by a compactly supported Hamiltonian system.

An interesting question first resolved by \cite{seidel_thesis} is whether there are exotic symplectomorphisms which cannot be isotoped to the identity through symplectomorphisms (and which are smoothly isotopic to the identity); this problem is known as the \emph{symplectic isotopy problem}. His construction involves a \emph{generalized Dehn twist} which is an exotic element of $\mathrm{Symp}_{c}(T^{*}S^{2})$, and he proved (the square) of his Dehn twist is smoothly isotopic but not symplectically isotopic to the identity element. Since Seidel's element is compactly supported, it can be implanted inside any symplectic $4$-manifold which contains a Lagrangian $2$-sphere.

In our framework, the square of Seidel's element lives in $\pi_{0}(\mathrm{Ham}_{c}(T^{*}S^{2}))$ and is the element corresponding to the $1$-periodic Reeb flow under the connecting homomorphism \eqref{eq:connecting-homo}. Interestingly enough, \cite{seidel_thesis} explains that the notion of a symplectic Dehn twist goes back to Arnol'd \cite{arnold-symplectic-monodromy}; see \cite{seidel_thesis,frauenfelder-schlenk-dehn-seidel-2005,uljarevic_floer_homology_domains}.

Particular interest has been paid to inextensible loops arising from autonomous loops (note that autonomous positive loops are precisely the $1$-periodic Reeb flows); the associated element of $\pi_{0}(\mathrm{Ham}_{c}(W))$ is known in the literature as a \emph{fibered Dehn twist}. This construction can be found in \cite[\S4]{seidel-graded-lagrangians} (he does not use the name fibered Dehn twists, and attributes the question to Eliashberg and Polterovich). The name ``fibered Dehn twist'' and the approach in terms of the Serre fibration property for the ideal restriction is due to unpublished work of Biran and Giroux. In \cite{seidel-graded-lagrangians,uljarevic_floer_homology_domains,chiang-ding-van-koert,merry_ulja,ulja_drobnjak} it is shown (in certain cases) that fibered Dehn twists are non-trivial in $\pi_{0}(\mathrm{Symp}_{c}(W))$; sometimes, it is even shown that the fibered Dehn twists are of infinite order in $\pi_{0}(\mathrm{Symp}_{c}(W))$.

Theorem \ref{theorem:main-absolute} implies that any positive loop in $\mathrm{Cont}_{0}(Y)$ is inextensible, provided $\mathrm{SH}(W)\ne 0$. Since positive loops remain positive under iteration, one concludes that, if $\mathrm{SH}(W)\ne 0$, then \emph{every positive loop} induces a mapping class of infinite order in $\pi_{0}(\mathrm{Ham}_{c}(W))$.

Moreover, the atoroidal assumption of our paper allows us to upgrade this to the following result:
\begin{theorem}\label{theorem:exotic-mapping-class}
  Let $W$ be an atoroidal and convex-at-infinity symplectic manifold, and suppose $\mathrm{SH}(W)\ne 0$. Then every positive loop in $\mathrm{Cont}(Y)$ (based at the identity) induces a mapping class of infinite order in $\pi_{0}(\mathrm{Symp}_{c}(W))$.
\end{theorem}
The trick used to upgrade non-triviality in $\pi_{0}(\mathrm{Ham}_{c}(W))$ to non-triviality in $\pi_{0}(\mathrm{Symp}_{c}(W))$ uses the notion of \emph{flux} (see \cite{mcduff-flux} and \cite[\S10.2]{mcduffsalamon-alt}), and the fact that the \emph{flux group} vanishes for atoroidal manifolds; see \S\ref{sec:gener-past-ator} for further discussion.

Theorem \ref{theorem:exotic-mapping-class} is proved in \S\ref{sec:exot-sympl-flux}.

\subsubsection{Exotic Lagrangian fillings associated to inextensible loops}
\label{sec:exot-fillings-lagr}

There is a Lagrangian analogue of the discussion in \S\ref{sec:exot-sympl}, and we briefly describe the construction.

As shown in \S\ref{sec:serre-fibr-prop}, the ideal restriction is a Serre fibration from the space of all contact-at-infinity Lagrangians, $\mathrm{Lag}(W)$, to the space of Legendrians, $\mathrm{Leg}(Y)$.

Let $\mathrm{Lag}(W;\Lambda)$ be the fiber over $\Lambda$ of the ideal restriction map $\mathrm{Lag}(W)\to \mathrm{Leg}(Y)$. Fix a particular element $L\in \mathrm{Lag}(W;\Lambda)$. The Serre fibration property induces a connecting homomorphism:
\begin{equation}\label{eq:lag-sfp}
  \pi_{1}(\mathrm{Leg}(Y),\Lambda)\to \pi_{0}(\mathrm{Lag}(W;\Lambda)),
\end{equation}
as usual: lift any loop $\Lambda_{t}$ based at $\Lambda$ to a path based at $L$ and then evaluate it at the endpoint of the path; see \S\ref{sec:lagrangian-ideal-restriction} for further details. Let us denote by: $$[L,\Lambda_{t}]\in \pi_{0}(\mathrm{Lag}(W;\Lambda))$$ the image of $[\Lambda_{t}]\in \pi_{1}(\mathrm{Leg}(Y),\Lambda)$ under this morphism; it is a relative version of the symplectic mapping classes considered above. Let us say that $[L,\Lambda_{t}]$ is \emph{trivial} if it is the component of $\pi_{0}(\mathrm{Lag}(W;\Lambda))$ containing $L$.

Saying that $\Lambda_{t}$ is \emph{extensible} (relative $L$) implies that $[L,\Lambda_{t}]$ is trivial. Under certain assumptions the converse holds:
\begin{lemma}\label{lemma:technical-lag-flux}
  If $L$ is acylindrical and every closed compactly supported one-form on $L$ can be extended to a closed compactly supported one-form on $W$, then triviality of $[L,\Lambda_{t}]$ implies $\Lambda_{t}$ is extensible.
\end{lemma}
There is a subtlety concerning the exactness of the loop which is resolved by considering a Lagrangian version of flux; see \S\ref{sec:lagrangian-flux} for further details. Unfortunately, we require the additional assumption on the compactly supported one-forms; see \cite{ono-lagrangian-flux,solomon-lagrangian-flux} for similar assumptions when considering Lagrangian flux. The authors are unsure whether or not the conditions on the closed one-forms can be removed. Our results yield:
\begin{theorem}\label{theorem:lagrangian-exotic}
  If $L\in \mathrm{Lag}(W;\Lambda)$ satisfies $\mathrm{HW}(L)\ne 0$, the hypotheses of Lemma \ref{lemma:technical-lag-flux} are satisfied, and $\Lambda_{t}$ is a positive loop based at $\Lambda$, then the isotopy classes $[L,\Lambda_{kt}]$ are all distinct for $k\in \Z$. Consequently, $\pi_{0}(\mathrm{Lag}(W;\Lambda))$ will contain infinitely many elements.
\end{theorem}
Indeed, one shows that $[L,\Lambda_{k t}]=[L,\Lambda_{\ell t}]$ if and only if $[L,\Lambda_{(k-\ell)t}]$ is trivial; if $k>\ell$ the latter follows from the Lemma \ref{lemma:technical-lag-flux} and Theorem \ref{theorem:main-relative}. This completes the proof.

Theorem \ref{theorem:lagrangian-exotic} is consistent with the famous result of \cite{eliashberg-polterovich-2-knots} which states that $\R^{2}\subset\R^{4}$ is the unique filling of its ideal restriction, since $\R^{2}\subset \R^{4}$ has vanishing wrapped Floer cohomology. It is also interesting to note that we do not recover the results of \cite{casals-gao,casals-ng}; these results use a finer algebraic structure and produce exotic fillings of certain Legendrian knots in the standard sphere $S^{3}$ (note that all such fillings will have vanishing wrapped Floer cohomology).

\subsubsection{Do Liouville manifolds with non-zero but finite-dimensional symplectic cohomology exist?}
\label{sec:open-quest-liouv}

It seems to be an open question whether there is a Liouville manifold with finite-dimensional and non-zero symplectic cohomology. Our result shows that the existence of an extensible positive loop of contactomorphisms implies the vanishing of symplectic cohomology. Having geometric criteria which ensure vanishing of $\mathrm{SH}$ simplifies the search for Liouville manifolds with finite-dimensional and non-zero symplectic cohomology, via process of elimination.

In \cite{ritter_negative_line_bundles,ritter_circle_actions}, examples of convex-at-infinity symplectic manifolds $W$ (containing symplectic spheres) are constructed so that $\mathrm{SH}(W)$ is a \emph{non-zero} quotient of the quantum cohomology ring $\mathrm{QH}(W)$, defined as $\mathrm{HF}(R^{\alpha}_{\epsilon t})$ for small $\epsilon>0$ for the purposes of this discussion; see \S\ref{sec:prequantization-bundles} for related discussion.

In the presence of holomorphic spheres, $\mathrm{QH}(W)$ and $\mathrm{SH}(W)$ are defined over the Novikov field of semi-infinite sums (over $\Z/2$) generated by symbols $\tau^{A}$ where $A\in \R$; the sums are semi-infinite in the sense that only finitely many terms in $\sum_{i=1}^{\infty} \tau^{A_{i}}$ should have $A_{i}$ less than $k$, for every $k$. Then $\mathrm{QH}(W)$ is finite dimensional over this field. Therefore the examples in \cite{ritter_negative_line_bundles,ritter_circle_actions} have $\mathrm{SH}(W)$ finite-dimensional and non-zero.

The arguments in our proof of Theorem \ref{theorem:main-absolute} break down in the presence of holomorphic spheres. Indeed, a key step in the proof is to find some system $\varphi_{t}$ that $\mathrm{HF}(\varphi_{t})\to \mathrm{SH}(W)$ is \emph{not} surjective. Here $\mathrm{HF}(\varphi_{t})$ is the Floer cohomology of a contact isotopy; see \S\ref{sec:floer-cohom-assoc}. When there are no quantum corrections, one can take a small negative Reeb flow $\varphi_{t}=R_{-\epsilon t}^{\alpha}$ in order to make $\mathrm{HF}(\varphi_{t})\to \mathrm{SH}(W)$ non-surjective; this is the content of Lemma \ref{lemma:non-surjective-lemma}. Essentially the idea is to prove the elements in the image of $\mathrm{HF}(R_{-\epsilon t}^{\alpha})\to \mathrm{HF}(R_{+\epsilon t}^{\alpha})$ are nilpotent with respect to the pair-of-pants product. This nilpotency may no longer be the case with respect to a quantum cup product.

The authors suspect that the following dichotomy should continue to hold in the presence of holomorphic spheres: either $\mathrm{HF}(\varphi_{t})\to \mathrm{SH}(W)$ is surjective \emph{for every} contact isotopy $\varphi_{t}$, or $\mathrm{SH}(W)$ is infinite dimensional.

Let us introduce $\mathrm{SH}^{e}(W)$ to be the subspace of $\mathrm{SH}(W)$ generated by those elements which lie in the image of $\mathrm{HF}(\varphi_{t})\to \mathrm{SH}(W)$ for every system $\varphi_{t}$; we call elements in this subspace ``eternal'' since they are never ``born.'' In the language of the persistence module considered in \cite{cant_sh_barcode} the eternal elements are generated by the fully infinite bars $(-\infty,\infty)$ while the quotient $\mathrm{SH}(W)/\mathrm{SH}^{e}(W)$ is generated by the half-infinite bars $[a,\infty)$.

Interestingly enough, it is precisely the elements of $\mathrm{SH}/\mathrm{SH}^{e}$ which are used to construct \emph{spectral invariants}; see \cite[\S1.2.5]{cant_sh_barcode}.

In the case of an aspherical manifold $W$, the arguments in \S\ref{sec:pss-comparison} can be performed and one can show that $\mathrm{SH}/\mathrm{SH}^{e}=0$ if and only if $\mathrm{SH}=0$ (one considers when the unit is born). The above alternative can therefore be conjectured as:
\begin{conjecture}\label{conj:alternative}
  For a convex-at-infinity symplectic manifold $W$, $$\mathrm{SH}(W)/\mathrm{SH}^{e}(W)$$ is either zero or infinite dimensional.
\end{conjecture}

Of course, in order for this conjecture to make sense, one needs to define $\mathrm{SH}(W)$ in the presence of holomorphic spheres; see \S\ref{sec:gener-past-ator} for further discussion.

The methods in \S\ref{sec:the-twist-trick} continue to apply even in the presence of holomorphic spheres (assuming semi-positivity or some virtual techniques). Thus one obtains: if the ideal boundary of $W$ admits an extensible positive loop, then $\mathrm{SH}(W)=\mathrm{SH}^{e}(W)$, i.e., the alternative in Conjecture \ref{conj:alternative} holds for manifolds admitting extensible positive loops.

It is perhaps interesting to compare this with \cite[\S4]{albers-kang-vanishing-RFH}, which discusses the vanishing of RFH in relation to Ritter's non-vanishing results for symplectic cohomology \cite{ritter_negative_line_bundles,ritter_circle_actions}; see also \cite{ritter-smith-open-closed} for a non-vanishing result for symplectic cohomology of monotone negative line bundles over closed monotone toric manifolds; here monotone means $[c_{1}]=\lambda[\omega]$ for some $\lambda>0$.

\subsubsection{Generalization past atoroidal}
\label{sec:gener-past-ator}

The atoroidal assumption is used in a few places in our paper. First, it is used to define $\mathrm{SH}(W)$ over the field $\Z/2$, rather than over a Novikov field. It seems that the aspherical assumption is not enough to ensure that $\mathrm{SH}(W)$ can be defined over $\Z/2$ because we require working with all free homotopy classes of orbits when considering naturality transformations.

It appears that one can relax the atoroidal/aspherical assumption at the expense of working over Novikov fields and assuming some sort of semi-positivity as in \cite{ritter_negative_line_bundles,ritter_circle_actions} (or using some sort of virtual perturbation techniques). Such extensions are left to further research, e.g., \cite{dusan_in_preparation}.

Another place where the atoroidal assumption seems more essential is when considering the flux associated to symplectic isotopies which are contact-at-infinity; it is used to argue that the flux of any loop is zero.

After the first version of this paper was posted, the authors learned\footnote{The authors wish to thank M.~Atallah and D.~ Rathel-Fournier for pointing this out.} that \emph{toroidally montone} manifolds $W$ have a vanishing flux group (i.e., the symplectic form is proportional to $c_{1}(W)$ when integrated over any smooth 2-tori). The reason is fairly straighforward: $TW$ can be symplectically trivialized when pulled back to the particular 2-tori which relevant to the computation of the flux of a loop; see also \cite[Theorem 1]{mcduff-flux} for a more general statement. There is also an analogue of \emph{cylindrically monotone} Lagrangians $L$, where the symplectic form is proportional to the Maslov class, and such Lagrangians have a vanishing flux group, in the sense that any loop based at $L$ has vanishing flux.

\subsection{Examples of extensible and inextensible positive loops}
\label{sec:exampl-extens-inext}

We discuss various examples of convex-at-infinity symplectic manifolds admitting extensible and inextensible positive loops.

\subsubsection{Negative line bundles}
\label{sec:prequantization-bundles}

In this section we show that prequantization bundles can be filled in a convex-at-infinity way. The construction provides examples of atoroidal fillings that are not Liouville fillings. The manifolds we construct are total spaces of negative line bundles as in \cite{ritter_negative_line_bundles,ritter_circle_actions,albers-kang-vanishing-RFH,venkatesh_quantitative_nature,albers_kang}.

Let $\pi:Y\to B$ be a principal $S^{1}$ bundle whose Euler class is represented by a symplectic form on $B$. Fix then a symplectic form $\omega$ so that:
\begin{equation*}
  [\omega]=-2\pi(\text{Euler class}).
\end{equation*}
This bundle admits a connection $1$-form $\alpha$ which is also a contact form whose Reeb flow generates the principal $S^1$-action (with period $2\pi$); see, e.g., \cite[Theorem 7.2.4]{Geiges}. As part of the construction it is shown that  $\pi^{*}\omega=\d\alpha$.

The $S^{1}$-bundle $\pi\colon Y\rightarrow B$ is the unit circle bundle associated to a complex line bundle $\pi\colon W\rightarrow B$. The total space $W$ can be naturally considered as a convex-at-infinity symplectic manifold with ideal boundary $(Y,\ker \alpha)$ as follows. A fiberwise unitary metric on $W$ gives rise to a radial coordinate $r$. The space $Y$ is identified with the set $\{r=1\}\subset W$.
The closed $2$-form:
$$\Omega:=\pi^{\ast}\omega+\d(r^2\alpha)$$
is a symplectic form on $W$; indeed, one can write it as $(1+r^{2})\pi^{*}\omega+2r\d r\wedge \alpha$, which is easily seen to be non-degenerate everywhere; see also \cite[pp.\,656]{mcduff-contact-boundaries}.

Note that away from the zero-section the form $\Omega$ is given by:
$$\Omega=\d((1+r^2)\alpha).$$

In particular $(1+r^2)\alpha$ is a Liouville form and induces the contact structure $\ker\alpha$ on $B$.
Because the zero-section of $W$ is a closed symplectic submanifold, $\Omega$ does not admit a Liouville form on the entire total space $W$.

If $B$ is symplectically atoroidal then so is $W$. Indeed, any torus in $W$ is homotopic to a torus in the zero-section, which is a submanifold symplectomorphic to $B$.

Similarly to case of the stabilizations considered in \S\ref{sec:stab-liouv-manif}, the Hamiltonian: $$H=2\pi r^2$$ generates an extensible positive loop (the vector field generating the ideal restriction is $2\pi R_{\alpha}$).

Consequently, if $B$ is atoroidal, Theorem \ref{theorem:main-absolute} ensures that the symplectic cohomology of $W$ vanishes. This recovers a special case of the results of \cite{oancea_fibered_sh} and \cite[Corollary 2]{ritter_negative_line_bundles}; the latter result states that the symplectic cohomology of the total space of a negative line bundle vanishes if $c_{1}(Y\to B)$ is torsion in the quantum cohomology ring of $W$ (which certainly holds when $W$ is atoroidal, since the quantum cohomology agrees with the usual cohomology); see also \cite{ritter_circle_actions,venkatesh_quantitative_nature,bae-kang-kim}.

Interestingly, \cite[Corollary 2]{ritter_negative_line_bundles} implies that $\mathrm{SH}(W)$ does \emph{not} vanish if the class $c_{1}(Y\to B)$ is not torsion in the quantum cohomology ring of $W$. Of course, in order to define $\mathrm{SH}(W)$ in this level of generality, one needs to deal with the effect of holomorphic spheres; see \S\ref{sec:gener-past-ator}. This implies that Theorem \ref{theorem:main-absolute} does not hold for arbitrary convex-at-infinity $W$ and the presence of holomorphic spheres ruins the vanishing result; see \S\ref{sec:open-quest-liouv} for further discussion.

It is also interesting to compare this vanishing result with the Rabinowitz Floer homology invariants for prequantization bundles constructed in \cite{albers_shelukhin_zapolsky,albers_kang}. See \cite{albers-kang-vanishing-RFH} which proves such a vanishing result for RFH for negative line bundles, (even in the presence of holomorphic spheres).

\subsubsection{Magnetic flows as extensible loops on the cotangent bundle of the two-sphere}
\label{sec:magnetic-flows-as}
It is well known that the symplectic cohomology of a cotangent bundle is isomorphic to the singular homology of the free loop space of the base and in particular is non-zero; see, e.g., \cite{abbondandolo-schwarz,abouzaid_monograph}.
Our results then imply that a periodic Reeb flow on the spherical cotangent bundle is never extensible in the standard cotangent bundle.
Nevertheless, in certain cases, the symplectic form can be twisted with a magnetic term to provide a filling of the spherical cotangent bundle for which the periodic Reeb flow \emph{is} extensible.

Following \cite[\S2]{seidel_thesis} and \cite[\S2a]{seidel-lecture-4d-dehn-twist}, we consider here the case of $T^{*}S^{2}$ and construct a symplectic form $\Omega$ on $T^{\ast}S^2$ which coincides with the standard symplectic form outside of a compact set and provides a filling of $ST^{\ast}S^2$ for which the standard periodic Reeb flow is extensible.

Let $\sigma$ be the area form and let $r$ be the radial coordinate on $T^{*}S^{2}$ associated to the round metric on $S^{2}$. Let $\Sigma(r_{0})=\set{r=r_{0}}$. Since $\mathrm{SO}(3)$ acts on $S^{2}$ by isometries, the associated canonical transformations induce an $\mathrm{SO}(3)$ action on $T^{*}S^{2}$ which is free and transitive on each hypersurface $\Sigma(r_{0})$.

Denote by $\pi:\Sigma(1)\to S^{2}$ the projection map. Since $\Sigma(1)$ admits $S^{3}$ as a covering space, it has trivial second deRham cohomology, and hence $\pi^{*}\sigma=\d\Pi$ for some one-form $\Pi$ on $\Sigma(1)$. Moreover, by a standard averaging argument, we may suppose that $\Pi$ is $\mathrm{SO}(3)$-invariant.

Since $\Pi$ is not pulled back from the base (as $\sigma$ is not exact), $\Pi$ is non-vanishing on the circular cotangent fibers, and hence $\Pi\wedge \d \Pi\ne 0$, i.e., $\Pi$ actually defines a contact structure on $\Sigma(1)$; indeed, one can realize $\Pi$ as a prequantization contact form of the kind considered in \S\ref{sec:prequantization-bundles}; see \cite{albers-geiges-zehmisch,allout-saglam} for further discussion.

Extend $\Pi$ to the complement of the zero section, using the fiberwise radial projection map.
Note that $\d\Pi=\pi^{\ast}\sigma$ extends smoothly to the zero section.

Let $f(r)$ be a smooth function so $f(r)=1$ for $r$ near zero, $f(r)=0$ for $r$ sufficiently large, and suppose $-\delta< f(r)f'(r)< \delta$, for some small number $\delta$. Then define:
\begin{equation}\label{eq:magnetic-symplectic-form}
  \Omega:=\d(\lambda+f(r)\Pi),
\end{equation}
where $\lambda$ denotes the standard Liouville form on $T^{\ast}S^2$.

\begin{lemma}\label{lemma:non-deg-magnetic}
The two-form $\Omega$ is a symplectic form on $T^{\ast}S^2$ that coincides with the standard symplectic form when $r$ is sufficiently large, provided $\delta$ is sufficiently small. In particular $(T^{\ast}S^2,\Omega)$ is a filling of the standard spherical cotangent bundle of $S^2$.
\end{lemma}
The argument is straightforward and is given in \S\ref{sec:hamilt-s1-acti}.

\begin{lemma}\label{lemma:magnetic-circle-action}
There exists a Hamiltonian circle action on $(T^{\ast}S^2,\Omega)$ whose ideal restriction is the standard periodic Reeb flow on $ST^{\ast}S^2$, provided $\delta$ is small enough.
\end{lemma}
See \cite{seidel_thesis,seidel-lecture-4d-dehn-twist} for a proof based on the moment map associated to the $\mathrm{SO}(3)$ action; see also \cite{bimmermann-2023} for an explicit construction of a Hamiltonian function which generates an $S^{1}$ action in a similar context. Our approach is based on an analysis of the characteristic foliations of the surfaces $\Sigma(r)$ as $r$ varies. The proof is given in \S\ref{sec:hamilt-s1-acti}.

To complete this section, we explain why our proof of Theorem \ref{theorem:main-absolute} can be adapted to show that $\mathrm{SH}(T^{*}S^{2},\Omega)$ is well-defined and vanishes, recovering a result of \cite{groman-merry-v4,benedetti-ritter}.

The crux of the matter is that $c_{1}(T^{*}S^{2},\Omega)=0$ (this follows from the case when $\Omega=\d\lambda$, and by continuity of the first Chern class with respect to deformations of the symplectic structure). This observation implies that \emph{generically there are no non-constant $J$-holomorphic spheres}; see \cite[Proposition 2.3.(ii)]{hofer-salamon-95}, recalling that $T^{*}S^{2}$ is a four-manifold. See also \cite{ritter_negative_line_bundles,mclean-ritter-mckay-correspondence} for further discussion of the consequences of the vanishing of $c_{1}$.

In particular, the arguments which rely on the aspherical assumption to avoid holomorphic spheres will go through. We should note, however, that we require working over the Novikov field, as explained in \S\ref{sec:gener-past-ator}.

\subsubsection{Divisor complements}
\label{sec:divisor-complements}

Let us call any closed submanifold which is Poincaré dual to a positive multiple of the symplectic form a \emph{symplectic divisor}; such divisors exist in every symplectic manifold (with an integral symplectic form) by work of \cite{donaldson96}. It is immediate that the complement of a divisor is exact. Moreover, the normal form for neighborhoods of symplectic submanifolds implies it has a contact type boundary; see \cite[Lemma 2.6]{mcduff-contact-boundaries}. Thus a divisor complement can be completed to an atoroidal convex-at-infinity manifold.\footnote{Moreover, the divisor complement is actually Liouville, i.e., the primitive in the end extends to the filling; see \cite{opshtein-JSG-2013,giroux-remarks-donaldsons,diogo-lisi-2019-jtopol}.} It can be seen that the normal form of the neighborhood of a symplectic divisor is as a neighborhood in a \emph{positive line bundle} (see \cite{mcduff-contact-boundaries,biran-lagrangian-barriers}), and consequently one can find a periodic Reeb flow on the boundary of the divisor complement; i.e., one can always find positive loops in the ideal boundaries of completed divisor complements.

In many cases it is known that completed divisor complements have non-vanishing symplectic cohomology. For instance, if $\Sigma\subset M$ is a symplectic divisor and $\omega$ vanishes on all spheres in $M$, then the linking Reeb orbit around $\Sigma$ is non-contractible in $M\setminus \Sigma$ and hence the unit in $\mathrm{SH}(W)$ is non-zero, where $W$ is the completion of $M\setminus \Sigma$. Another instance is when $\Sigma\subset \mathrm{CP}^{n}$ is the divisor cut out by $z_{0}^{2d}+\dots+z_{n}^{2d}=0$, since this divisor complement contains a weakly exact Lagrangian (namely $\mathrm{RP}^{n}$). It is a well-known result due to \cite{viterbo_functors_and_computations_1} that containing a weakly-exact Lagrangian implies non-vanishing of symplectic cohomology.

For divisors with non-vanishing symplectic cohomology, our result on symplectic mapping classes applies and we conclude the connecting homomorphism induces a mapping class of infinite order.

For other work on divisor complements and their symplectic cohomology, we refer the reader to: \cite{biran-lagrangian-barriers,diogo-thesis,albers_merry_orderability_non_squeezing_RFH,diogo-lisi-2019-jtopol,diogo-lisi-2019-jfpta,mclean20-annals,borman-sheridan-varolgunes,bae-kang-kim}

\subsubsection{Links of singularities of weighted homogeneous polynomials}
\label{sec:links-singularities}

The link of an isolated singularity of a weighted homogeneous polynomial admits a positive loop of contactomorphisms; see \cite[Proposition 2.5]{kwon-van-koert}. Moreover, assuming the isolated singularity at the origin is actually a critical point, the singularity has a positive Milnor number; see \cite{milnor68-singular-points,milnor-orlik}. Consequently, the filling has non-vanishing symplectic homology as proved in \cite[Theorem 6.3]{kwon-van-koert}; one shows this using the result of \cite{viterbo_functors_and_computations_1} which says convex-at-infinity manifolds containing weakly exact Lagrangians have non-vanishing symplectic cohomology; see also \cite[\S5]{seidel-biased}. Hence Theorem \ref{theorem:exotic-mapping-class} applies in this setting, and we conclude symplectic mapping classes of infinite order.

\subsection{Acknowledgements}
\label{sec:acknowledgements}

The authors wish to thank O.~Cornea, Y.~Eliashberg, E.~Shelukhin, and I.~Uljarevi\'c for their valuable guidance during the preparation of this paper. The authors also wish to thank M.~Atallah, D.~Chen, and D.~Rathel-Fournier for clarifying discussions after the first version of this paper was posted. The first two authors were supported in their research at Universit\'e de Montr\'eal by funding from the Fondation Courtois.

%%% \S2
\section{Symplectic cohomology and positive loops of contactomorphisms}
\label{sec:symplectic-cohomology-and-positive-loops}

Following \cite{merry_ulja,ulja_zhang,cant_sh_barcode,djordjevic_uljarevic_zhang}, to each system $\psi_{t}\in \mathrm{Ham}(W)$ whose ideal restriction has no discriminant points we associate a Floer cohomology group $\mathrm{HF}(\psi_{t})$; we describe the construction in \S\ref{sec:floer-cohom-groups}, following the conventions in \cite{cant_sh_barcode}. See \S\ref{sec:role-discr-points} for definition of discriminant point.

One aspect of our definition that merits mentioning is that $\mathrm{HF}(\psi_{t})$ depends only on the time-$1$ map $\psi_{1}$; however, various decorations such as supergradings and action filtrations do depend on the system $\psi_{t}$.

There are two important relationships involving the groups $\mathrm{HF}(\psi_{t})$. First, there are the \emph{continuation maps} associated to non-negative paths. Recall that a path $\psi_{s,t}$ is \emph{non-negative} if the ideal restriction of the time-1 maps $\psi_{s,1}$ is a non-negative path of contactomorphisms. In \S\ref{sec:sympl-cohom-as} we recall the definition of $\mathrm{SH}(W)$ as a colimit of groups $\mathrm{HF}(\psi_{t})$ with respect to continuation maps.

The second important relationship is the naturality transformation associated to a loop $\varphi_{t}$ in $\mathrm{Ham}(W)$ based at the identity; see \cite[\S4]{seidel_representation}, \cite[\S1.2]{ritter_negative_line_bundles} and \cite[\S2.7]{uljarevic_floer_homology_domains}. One observes that the systems $\psi_{t}$ and $\varphi_{t}\psi_{t}$ have the same fixed points at $t=1$ (since $\varphi_{1}=\id$), and so the chain complexes $\mathrm{CF}(\psi_{t})$ and $\mathrm{CF}(\varphi_{t}\psi_{t})$ are literally the same. This identity map is a chain map and the isomorphism $\mathrm{HF}(\psi_{t})\to \mathrm{HF}(\varphi_{t}\psi_{t})$ is called a \emph{naturality transformation}; see \S\ref{sec:natur-transf} for more details.

If the ideal restriction of $\varphi_{t}$ is a non-negative loop, then the path $\psi_{s,t}=\varphi_{st}\psi_{t}$ is a non-negative path and hence has an associated continuation map. The main strategy of this paper is to compare the two morphisms $\mathrm{HF}(\psi_{t})\to \mathrm{HF}(\varphi_{t}\psi_{t})$ (naturality versus continuation) in such a way as to conclude the colimit $\mathrm{SH}(W)$ vanishes; see \S\ref{sec:the-twist-trick}.

The argument makes use of the ring structures on symplectic and wrapped Floer cohomologies. In particular, the strategy is to show that the unit of the respective ring vanishes, and thus the ring is 0. In order to effectively implement this, it is also necessary to give a practical characterization of the unit, and we will follow \cite{ritter_TQFT}. In \S\ref{sec:PSS-and-the-unit-sympl-cohom} we give some brief exposition of this setup, and state the relevant result for symplectic cohomology.

\subsection{Fibration sequence of symplectomorphism groups}
\label{sec:fibr-sequ-sympl}

The goal in this section is to explain why the ideal restriction homomorphism $\mathrm{Ham}(W)\to \mathrm{Cont}(Y)$ is a Serre fibration. For further details, we refer the reader to \cite[\S2.3]{ulja_drobnjak} and \cite[pp.\,43]{merry_ulja}, both of which refer to an unpublished manuscript of Biran and Giroux; see also \cite{cool_gadget}. The paper \cite[\S3]{casals-keating-smith} constructs the connecting homomorphism as a symplectic \emph{Gromoll map} and uses this to study symplectic mapping class groups; see also \cite[\S{A}]{casals-keating-smith} for appendix by S.~Courte which describes the Serre fibration.

In \S\ref{sec:topol-group-cont} we define the topologies on $\mathrm{Ham}(W)$ and $\mathrm{Cont}(Y)$ and in \S\ref{sec:proof-serre-fibr-prop} we sketch the proof of the Serre fibration property.

\subsubsection{Topologizing the space of contact-at-infinity Hamiltonian diffeomorphisms}
\label{sec:topol-group-cont}

We topologize $\mathrm{Ham}(W)$ as follows. For each Hamiltonian diffeomorphism $\psi$ and compact set $K$ large enough that $\psi$ is equivariant outside of $K$, define $U_{\delta}(\psi,K)$ to be all other Hamiltonian diffeomorphisms which are equivariant outside of $K$ and which remain a distance at most $\delta$ from $\psi$ in some $C^{\infty}$ distance (using a metric which is translation invariant in the convex end).

Since $\mathrm{Ham}(W)$ only consists of the time-1 maps of contact-at-infinity systems, it follows that these open sets $U_{\delta}(\psi,K)$ cover $\mathrm{Ham}(W)$; moreover, they form a basis for a topology. Because we use a translation invariant metric, it follows that the ideal restriction map is continuous.

The topology has the property that, given a continuous map $x\mapsto \varphi_{x}\in \mathrm{Ham}(W)$ defined on a compact disk $D$, there exists a fixed compact set $K$ so that $\varphi_{x}$ is equivariant outside of $K$ for each $x\in D$; see \cite[\S2.2.3]{cant_sh_barcode} for related discussion.

\subsubsection{Proof of the Serre fibration property}
\label{sec:proof-serre-fibr-prop}

To show that the ideal restriction map \( \mathrm{Ham}(W) \to \mathrm{Cont}(Y) \) is a Serre fibration we must complete the lifting diagram:
\begin{equation*}
  \begin{tikzcd}
    {D^n} \arrow[r,"\Phi_{0}"] \arrow[d,hook,swap,"\mathrm{id}\times \set{0}"] & {\mathrm{Ham}(W)} \arrow[d]\\
    {D^n \times {[0,1]}} \arrow[r,"\phi"] \arrow[ur,dashed,"\Phi"] & {\mathrm{Cont}(Y)}
  \end{tikzcd}.
\end{equation*}
By inverting the time 0 map \( \phi(x,0) \) we obtain \( \tilde{\phi}(x,t) := \phi(x,0)^{-1}\circ \phi(x,t) \), a continuous map \( D^n \times [0,1] \to \mathrm{Cont}(Y) \). The well-known correspondence between contactomorphisms of $Y$ and equivariant symplectomorphisms of $SY$ yields a map \( \psi:D^n \times [0,1] \to \mathrm{Ham}(SY) \) whose ideal restriction is $\tilde{\phi}$. Let \( H_{x,t}:SY \to \mathbb{R} \) denote a (smooth) family of Hamiltonian functions generating \( \psi \).

As part of the data of a contact-at-infinity symplectic manifold \( W \) we fix a (symplectic) embedding of some half space within \( SY \) to \( W \). Using this embedding, we can push forward the Hamiltonians \( H_{x,t} \) to the end of $W$. Cutting off by a smooth function \( f:SY\to \R \) which vanishes on the negative end of our half space, and which is identically 1 outside of a compact set, we obtain, by flowing, a family \( \tilde{\Phi} \) of contact-at-infinity Hamiltonian symplectomorphisms of \( W \) whose ideal restriction is \( \tilde{\phi} \); it is important to note that the construction yields $\tilde{\Phi}_{x,0}=\id$. Then \( \Phi_{x,t} := \Phi_{0}(x)\circ \tilde{\Phi}_{x,t} \) is the required extension, completing the proof.

\subsection{Floer cohomology groups for contact-at-infinity Hamiltonian systems}
\label{sec:floer-cohom-groups}

In this section we define the Floer cohomology groups associated to contact-at-infinity Hamiltonian systems.

\subsubsection{Almost complex structures}
\label{sec:almost-compl-struct}

Let $\mathscr{J}$ be the space of almost complex structures which are $\omega$-tame and equivariant with respect to the Liouville flow, outside of a compact set; see \cite[\S2.1.3]{brocic_cant} and \cite[\S2.1.2]{cant_sh_barcode} for further discussion.

\subsubsection{On the role of discriminant points}
\label{sec:role-discr-points}

In order to define Floer cohomology of a Hamiltonian system $\psi_{t}$, it is necessary that $\psi_{1}$ has non-degenerate fixed points. If $\psi_{1}$ is equivariant with respect to the Liouville flow, outside of a compact set $K$, then every fixed point of $\psi_{1}$ which lies outside of $K$ automatically occurs in a one-parameter family (by translation via the Liouville flow). For this reason, we always assume that $\psi_{1}$ has no fixed points outside of a compact set. Interestingly enough, this condition is equivalent to the ideal restriction not having any \emph{discriminant points}; see \cite[\S1.1.2]{cant_sh_barcode}.

\subsubsection{Definition of the Floer cohomology group}
\label{sec:defin-floer-cohom}

Let us say that a pair $(\psi_{t},J_{t})$ with $\psi_{t}\in \mathrm{Ham}(W)$ and $J_{t}\in \mathscr{J}$ is \emph{admissible} provided the following conditions hold:
\begin{enumerate}
\item $J_{t+1}(w)=\d\psi_{1}^{-1} J_{t}(\psi_{1}(w))\d\psi_{1}$, i.e., $J_{t}$ is \emph{twisted periodic},
\item the ideal restriction of $\psi_{1}$ has no discriminant points,
\item all fixed points of $\psi_{1}$ are non-degenerate, and
\item the moduli space $\mathscr{M}(\psi_{t},J_{t})$ of finite-energy twisted holomorphic cylinders:
  \begin{equation*}
    \left\{
      \begin{aligned}
        &w:\C\to W,\\
        &\bd_{s}w+J_{t}(w)\bd_{t}w=0,\\
        &\psi_{1}(w(s,t+1))=w(s,t),
      \end{aligned}
    \right.
  \end{equation*}
  is cut transversally (in the usual Floer theoretic sense that the linearized operator is surjective at all solutions). Let us denote by $\mathfrak{A}^{\times}$ the space of admissible data, with $\mathfrak{A}$ referring to all choices of data satisfying (i).
\end{enumerate}

Here the \emph{energy} is the symplectic area of $w$ restricted to $\R\times [0,1]$.

See \cite[\S3]{dostoglou_salamon} for consideration of similar twisted holomorphic cylinders.

If $(\psi_{t},J_{t})\in \mathfrak{A}^{\times}$ then we define $\mathrm{CF}(\psi_{t},J_{t})$ to be the vector space over $\Z/2$ generated by the (finitely many) fixed points of $\psi_{1}$. The differential on $\mathrm{CF}(\psi_{t},J_{t})$ is defined, as usual, by counting elements in $\mathscr{M}_{1}(\psi_{t},J_{t})/\R$, where $\mathscr{M}_{1}$ denotes the one-dimensional component; see \cite[\S2.1.4]{cant_sh_barcode} for further discussion.

One should note that the choice of system $\psi_{t}$ enables us to canonically associate the transformation $u(s,t)=\psi_{t}(w(s,t))$ which satisfies: $$u(s,t+1)=\psi_{t+1}(w(s,t+1))=\psi_{t}(w(s,t))=u(s,t),$$ i.e., $u$ is defined on the cylinder $\R\times \R/\Z$. A standard computation shows that $u$ solves the normal Floer's equation with $X_{t}$ defined by $X_{t}\circ \psi_{t}=\frac{d}{dt}\psi_{t}$, with respect to some periodic $\bar{J}:\R/\Z\to \mathscr{J}$; indeed:
\begin{equation}\label{eq:bar-J-eq}
  \bar{J}_{t}(u)=\d\psi_{t}J_{t}(\psi_{t}^{-1}(u))\d\psi_{t}^{-1};
\end{equation}
see \cite[\S2.2]{cant_sh_barcode} for further details.

In this fashion, we see that the differential defined using the twisted holomorphic curves in $\mathscr{M}(\psi_{t},J_{t})$ is equivalent to one defined by counting solutions to Floer's equation on $\R\times \R/\Z$; see Figure \ref{fig:diff}.

\begin{figure}[H]
  \centering
  \begin{tikzpicture}
    \draw (0.25,0) arc (0:360:0.25 and 0.5) coordinate[pos=0.25](X)coordinate[pos=0.75](Y) node[pos=0.5,left]{output};
    \draw (8.25,0) arc (0:360:0.25 and 0.5) coordinate[pos=0.25](A)coordinate[pos=0.75](B) node[right]{input};
    \draw (X)--(A) (Y)--(B);
    \path (X)--(B) node[pos=0.5]{$\bd_{s}u+\bar{J}_{t}(u)(\bd_{t}u-X_{t}(u))=0$};
  \end{tikzpicture}
  \caption{Differential is defined by counting solutions to Floer's equation on the cylinder.}
  \label{fig:diff}
\end{figure}

The homology of $\mathrm{CF}(\psi_{t},J_{t})$ is denoted by $\mathrm{HF}(\psi_{t},J_{t})$, and is called the \emph{Floer cohomology} of the data $(\psi_{t},J_{t})$. By construction, it depends only on $(\psi_{1},J_{t})$, a fact which can be encoded in terms of \emph{naturality transformations}; see \S\ref{sec:natur-transf}.

Typically one defines decorations on $\mathrm{HF}(\psi_{t},J_{t})$ (e.g., action filtrations, the distinguished subcomplex generated by contractible orbits, gradings, etc) and these decorations a priori depend on the choice of system $\psi_{t}$.

\subsubsection{Continuation maps}
\label{sec:continuation-maps}

Let us call a path $(\psi_{s,t},J_{s,t})\in \mathfrak{A}$, for $s\in \R$, \emph{continuation data} provided:
\begin{enumerate}
\item the ideal restriction of $\psi_{s,1}$ is non-negative,
\item $\psi_{s,t},J_{s,t}$ are $s$-independent for $s$ outside of a compact interval $[s_{0},s_{1}]$,
\item $\psi_{s_{0},t},J_{s_{0},t}\in \mathfrak{A}^{\times}$ and $\psi_{s_{1},t},J_{s_{1},t}\in \mathfrak{A}^{\times}$, and,
\item $\psi_{s,t}=\psi_{s,1}$ holds for $t\in (2/3,1]$ and $\psi_{s,t}=\id$ holds for $t\in [0,1/3)$.
\end{enumerate}
For such data, we will define a continuation map $\mathrm{CF}(\psi_{s_{0},t},J_{s_{0},t})\to \mathrm{CF}(\psi_{s_{1},t},J_{s_{1},t})$. Condition (iv) can be achieved, without loss of generality, by a time reparametrization. We also extend $\psi_{s,t}$ to all $t\in \R$ by requiring $\psi_{s,t+1}=\psi_{s,t}\psi_{s,1}$.

Write $Y_{s,t}$ and $X_{s,t}$ for the infinitesimal generators of $\psi_{-s,t}$, i.e.,
\begin{equation*}
Y_{s,t}\circ \psi_{-s,t}=\bd_{s}\psi_{-s,t}\text{ and }X_{s,t}\circ \psi_{-s,t}=\bd_{t}\psi_{-s,t},
\end{equation*}
and observe that $Y_{s,t}=Y_{s,1}$ holds for $t\in (2/3,1]$, and $Y_{s,t}=0$ for $t\in [0,1/3)$, while $X_{s,t}=0$ for $t\in [0,1/3)\cup (2/3,1]$.

As in \cite[\S2.2.4]{cant_sh_barcode}, we define $\mathscr{M}(\psi_{s,t},J_{s,t})$ to be the moduli space of finite energy solutions to:
\begin{equation*}
  \left\{
    \begin{aligned}
      &u:\R\times \R/\Z\to W,\\
      &(\bd_{s}u-\rho(t)Y_{s,t})+\bar{J}_{-s,t}(u)(\bd_{t}w-X_{s,t})=0,
    \end{aligned}
  \right.
\end{equation*}
where $\rho(t)$ is a smooth cut off function so that $\rho(t)=1$ for $t\le 2/3$ and $\rho(t)=0$ for $t\ge 1$, and $\bar{J}_{s,t}(u)=\d\psi_{s,t}J_{s,t}(\psi_{s,t}^{-1}(u))\d\psi_{s,t}^{-1}$ as in \eqref{eq:bar-J-eq}.

It is important to note that, since $Y_{s,t}=0$ and $\bd_{s}X_{s,t}=\bd_{s}J_{s,t}=0$ holds whenever $s\not\in (s_{0},s_{1})$, solutions to $\mathscr{M}(\psi_{s,t},J_{s,t})$ are asymptotic to solutions of the $s$-independent Floer's equation at its ends. We put the minus sign in $-s$ so that $u$ is asymptotic at its right end to an orbit of the system $\psi_{s_{0},t}$, because the right end is supposed to be the input to the continuation morphism.

Let us say that continuation data $(\psi_{s,t},J_{s,t})$ is \emph{admissible} continuation data provided the moduli space $\mathscr{M}(\psi_{s,t},J_{s,t})$ is cut transversally. Admissibility can be achieved by perturbing $\psi_{s,t}$, $J_{s,t}$ where $s\in (s_{0},s_{1})$ and $t\in (1/3,2/3)$, without changing the ideal restriction.

Similarly to \cite[\S2.2.4]{cant_sh_barcode}, the non-negativity assumption implies an a priori energy bound, which is used to establish compactness-up-to-breaking of the relevant moduli spaces; see \S\ref{sec:energy-estim-ator}. The result is that one obtains a chain map:
\begin{equation*}
  \mathfrak{c}:\mathrm{CF}(\psi_{s_{0},t},J_{s_{0},t})\to \mathrm{CF}(\psi_{s_{1},t},J_{s_{1},t}).
\end{equation*}
Standard arguments imply the chain homotopy class of $\mathfrak{c}$ depends only on the homotopy class of $\psi_{s,1}$ in the space of non-negative paths with fixed endpoints; see \cite[\S2.2.8]{cant_sh_barcode} and \cite[Lemma 6.13]{abouzaid_monograph}.

\subsubsection{Energy estimates and the atoroidal condition}
\label{sec:energy-estim-ator}

Let $(\psi_{s,t},J_{s,t})$ be admissible continuation data, so that $\psi_{s,1}$ is a non-negative path. The goal in this section is to prove elements in $\mathscr{M}(\psi_{s,t},J_{s,t})$ satisfy an a priori energy bound.

Given $u\in \mathscr{M}(\psi_{s,t},J_{s,t})$. The energy of $u$ is defined to be the quantity
\begin{equation*}
  E(u)=\int_{\R\times \R/\Z}\omega(\bd_{s}u-\rho(t)Y_{s,t}(u),\bd_{t}u-X_{s,t}(u))\d s\d t;
\end{equation*}
because $\bar{J}_{s,t}$ is $\omega$-tame, the integrand is everywhere non-negative.

\begin{lemma}
  There exists a constant $C=C(\psi_{s,t},J_{s,t})$ so that $E(u)\le C$ holds for all solutions $u\in \mathscr{M}(\psi_{s,t},J_{s,t})$; moreover the constant $C$ can be taken continuous with respect to compactly supported perturbations of $\psi_{s,t},J_{s,t}$.
\end{lemma}
\begin{proof}
  Let $H_{s,t}$ and $K_{s,t}$ be the normalized\footnote{We say that $H\in C^{\infty}(W)$ is normalized if $H$ is eventually one-homogeneous in a distinguished connected component of the ideal boundary; this depends on an auxiliary choice of distinguished component, which we fix once and for all. Contact-at-infinity systems can always be generated by time-dependent normalized Hamiltonians. Normalized functions which are constant vanish identically, since we assume $W$ is connected.} generators for $X_{s,t},Y_{s,t}$. A direct computation shows that:
  \begin{equation*}
    \begin{aligned}
      E(u)&=\int u^{*}\omega+\int \rho(t)\d K_{s,t}(\bd_{t}u)-\d H_{s,t}(\bd_{s}u)+\rho(t)\omega(Y_{s,t},X_{s,t})\d s\d t\\
          &=\int u^{*}\omega+\int_{\R/\Z} H_{-,t}(\gamma_{-})\d t-\int_{\R/\Z} H_{+,t}(\gamma_{+})\d t-\int \rho'(t)K_{s,t}(u)\d s\d t+r,
    \end{aligned}
  \end{equation*}
  where $H_{s,t}=H_{\pm,t}$ holds for $\pm s$ sufficiently large, $\gamma_{\pm}$ are the asymptotics orbits of $u$, and:
  \begin{equation*}
    r=\int\bd_{s}H_{s,t}-\rho(t)\bd_{t}K_{s,t}+\rho(t)\omega(Y_{s,t},X_{s,t})\d s\d t
  \end{equation*}
  should be considered as a sort of contribution to the ``curvature term'' (in the sense of Hamiltonian connections; see \cite[\S8]{mcduffsalamon}). It is a fact that:
  \begin{equation}\label{eq:curvature-no-rho}
    \bd_{s}H_{s,t}-\bd_{t}K_{s,t}+\omega(Y_{s,t},X_{s,t})=0
  \end{equation}
  holds pointwise; indeed, this can be proved by differentiating $f\circ \psi_{-s,t}$ with respect to $\bd_{s}\bd_{t}$ and $\bd_{t}\bd_{s}$ and using Clairaut's theorem on the equality of mixed partial derivatives. One shows that, for $s,t$ fixed, and where $F$ is the Hamiltonian vector field for $f$, that:
  \begin{equation*}
    F\intprod \d(\bd_{s}H-\bd_{t}K)=Y\intprod\d(X\intprod \d f)-X\intprod \d(Y\intprod \d f)=F\intprod \d\omega(X,Y),
  \end{equation*}
  and hence $\bd_{s}H_{s,t}-\bd_{t}K_{s,t}-\omega(Y_{s,t},X_{s,t})$ has vanishing derivative (for $s,t$ fixed). This constant is in fact zero, because we suppose $H,K$ are the normalized generators.

  Since $X_{s,t}=0$ and $Y_{s,t}=Y_{s,1}$ (and hence $K_{s,t}=K_{s,1}$) holds for $t\ge 2/3$, and $\rho(t)=1$ for $t\le 2/3$, it follows from \eqref{eq:curvature-no-rho} that $r=0$. Consequently:
  \begin{equation*}
    E(u)=\int u^{*}\omega+\int_{\R/\Z} H_{-,t}(\gamma_{-})\d t-\int_{\R/\Z} H_{+,t}(\gamma_{+})\d t-\int \rho'(t)K_{s,1}(u)\d s\d t,
  \end{equation*}
  where we use that $\rho'(t)K_{s,t}=\rho'(t)K_{s,1}$. Each term in this formula can be bounded independently of $u$. First, the symplectic area of the cylinder is bounded because of the symplectically atoroidal condition. Second, the integrals of $H_{\pm,t}$ are bounded in terms of the maximums of $H_{\pm}$ on any compact set which contains all asymptotic orbits (bearing in mind there are no orbits at infinity). Finally, the generator $K_{s,1}$ is \emph{non-positive} outside of a compact set, since it is the generator for $\psi_{-s,1}$, and hence the final term can be bounded by the (finite) maximum of $\rho'(t)K_{s,1}$. This completes the proof of the a priori energy bound.  
\end{proof}

A priori estimates on $C^{0}$ and $C^{1}$ are discussed in \S\ref{sec:energy-estim-other}; see also \cite{brocic_cant,cant_sh_barcode}.

\subsubsection{Naturality transformations}
\label{sec:natur-transf}

Let $\varphi_{t}$ be a loop in $\mathrm{Ham}(W)$ based at the identity. The goal in this section is to construct the naturality isomorphism:
\begin{equation*}
  \mathfrak{n}:\mathrm{HF}(\psi_{t})\to \mathrm{HF}(\varphi_{t}\psi_{t}).
\end{equation*}
We refer the reader to \cite[\S4]{seidel_representation} and \cite[\S2.7]{uljarevic_floer_homology_domains} for related discussion.

By definition, the map $\mathfrak{n}$ acts identically on $\mathrm{CF}(\psi_{t})=\mathrm{CF}(\varphi_{t}\psi_{t})$. Recalling that the differential counts solutions to the equation in \S\ref{sec:defin-floer-cohom}, we similarly conclude that the differentials are the same for $\mathrm{CF}(\psi_{t})=\mathrm{CF}(\varphi_{t}\psi_{t})$, and hence $\mathfrak{n}$ is a chain isomorphism.

Similarly, if $\psi_{s,t}$ is any path of systems to that $s\mapsto \psi_{s,1}$ is non-negative, then the continuation map associated to $\varphi_{t}\circ \psi_{s,t}$ from $\mathrm{CF}(\varphi_{t}\psi_{0,t})$ to $\mathrm{CF}(\varphi_{t}\psi_{1,t})$ commutes with the continuation map $\mathrm{CF}(\psi_{0,t})\to \mathrm{CF}(\psi_{1,t})$ with respect to the naturality transformations; indeed, as above, the moduli spaces used are literally the same, since they depend only on the restriction to $t=1$.

\subsubsection{Digression on the definition of a colimit}
\label{sec:defin-colim}

We digress for a moment on the category theoretic definition of colimit. If $F:A\to B$ is a functor, consider the data of (i) an object $b\in B$ and (ii) a morphism $s_{a}:F(a)\to b$ for each $a\in A$, so that the following triangle commutes for all $\mu:a\to a'$:
\begin{equation*}
  \begin{tikzpicture}[xscale=1.25]
    \path (150:1)node(A){$F(a)$}--(30:1)node(B){$F(a')$}--(270:1)node(C){$b$};
    \draw[->] (A)--node[above,scale=0.75]{$F(\mu)$}(B);
    \draw[->] (A)--node[left,scale=0.75]{$s_{a}$}(C);
    \draw[->] (B)--node[right,scale=0.75]{$s_{a'}$}(C);
  \end{tikzpicture}
\end{equation*}
Given two such data $(b,s)$ and $(b',s')$, one considers morphisms $b\to b'$ so that the following triangle commutes for all objects $a\in A$:
\begin{equation*}
  \begin{tikzpicture}[xscale=1.25]
    \path (210:1)node(A){$b$}--(-30:1)node(B){$b'$}--(90:1)node(C){$F(a)$};
    \draw[->] (A)--(B);
    \draw[->] (C)--node[left,scale=0.75]{$s_{a}$}(A);
    \draw[->] (C)--node[right,scale=0.75]{$s_{a}'$}(B);
  \end{tikzpicture}
\end{equation*}
This defines an auxiliary category associated $F$ whose objects are data $(b,s)$ and whose morphisms are as above; this category is an example of a \emph{comma category}; see \cite{maclean-cat-working-math}.

If this auxiliary category has an initial object then we say the object is the \emph{colimit} of $F$. This is the universal property for the colimit. See \cite[\S I.5.1 and \S VIII.1.4]{aluffi} for more details. If such an object always exists for any choice of functor $F$, \( B \) is called \emph{cocomplete}.

It is a standard exercise in category theory to show that the category of vector spaces is cocomplete.

\subsubsection{Cofinal functors}
\label{sec:cofinal-sequences}

Let $B$ be a cocomplete category, and let $F:A\to B$ be a functor with colimit $\mathrm{colim}(F)$. Let \( C \) be small, and $\mathfrak{N}:C\to A$ be a functor; typically we will take $C=\mathbb{N}$ to be the category with a single morphism $n\to m$ whenever $n\le m$ (and zero morphisms otherwise).

By the universal property for the colimit, there is always a morphism:
\begin{equation}\label{eq:nat_map_colim}
  \mathrm{colim}(F \circ \mathfrak{N})\to \mathrm{colim}(F).
\end{equation}
In this section we give sufficient criteria on $\mathfrak{N}$ for \eqref{eq:nat_map_colim} to be an isomorphism. Such \( \mathfrak{N} \) are called \emph{cofinal}; see \cite[\S{IX}.3]{maclean-cat-working-math}.

Let us say that $\mathfrak{N}:C \to A$ is \emph{filtering} provided:
\begin{enumerate}
  \item[(i)] for all objects $a$, there exists some morphism $a\to \mathfrak{N}(c)$ for some $c\in C$,
  \item[(ii)] for every pair of morphisms \( f,g:a \to \mathfrak{N}(c) \) there is \( c' \in C \), \( h:c \rightarrow c' \) such that \( \mathfrak{N}(h)\circ f = \mathfrak{N}(h) \circ g \).
\end{enumerate}
The significance of this condition is that it gives a criterion for cofinality:

\begin{lemma}\label{lemma:filter-cofinal}
  Filtering functors are cofinal.
\end{lemma}
\begin{proof}
  The strategy is to produce an inverse to the natural map \eqref{eq:nat_map_colim} by applying the universal property to the maps yielded by the filtering property.

  By property (i), for any \( a \) there is a map \( F(g):F(a) \rightarrow F(\mathfrak{N}(c)) \) for some \( c \in C, \) and moreover by property (ii), the induced map \( F(a) \to \mathrm{colim}(F\circ \mathfrak{N}) \) is independent of the choice of map \( a\to \mathfrak{N}(c) \). We thus obtain a canonical map \( \mathrm{colim}(F) \to \mathrm{colim}(F\circ \mathfrak{N}), \) which by construction provides factorizations of the identity maps \( \mathrm{colim}(F) \to \mathrm{colim}(F), \) and \( \mathrm{colim}(F\circ \mathfrak{N})\to\mathrm{colim}(F\circ \mathfrak{N}) \). It follows that the map \eqref{eq:nat_map_colim} is an isomorphism.
\end{proof}

\subsubsection{Symplectic cohomology as a colimit}
\label{sec:sympl-cohom-as}

Consider the small category $\Delta$ whose objects are admissible data $(\psi_{t},J_{t})$ for defining the Floer complex, and whose morphisms $(\psi_{0,t},J_{0,t})\to (\psi_{1,t},J_{1,t})$ are homotopy classes of extensions $\psi_{s,t}$ so that $\psi_{s,0}=\id$ and the ideal restriction of $\psi_{s,1}$ is a non-negative path of contactomorphisms.

Define a functor from $\Delta$ into the category of $\Z/2$-vector spaces by assigning $(\psi_{t},J_{t})$ to the Floer cohomology $\mathrm{HF}(\psi_{t},J_{t})$, as defined in \S\ref{sec:defin-floer-cohom}. To each morphism (i.e., homotopy class of maps $\psi_{s,t}$), the continuation map associated to $\psi_{s,1}$ gives a morphism $\mathrm{HF}(\psi_{0,t},J_{0,t})\to \mathrm{HF}(\psi_{1,t},J_{1,t})$. By construction, and the fact continuation maps are invariant under homotopies, this prescription is functorial.

Define $\mathrm{SH}(W)$ to be the colimit of this functor $\mathrm{HF}:\Delta\to \mathrm{Vect}$; in \S\ref{sec:domin-sequ} we explain how to compute this colimit.

\subsubsection{The Floer cohomology associated to a contact isotopy}
\label{sec:floer-cohom-assoc}

Let $\zeta_{t}$ be a contact isotopy of $Y$ with $\zeta_{0}=\id$. One can consider the subcategory $\Delta(\zeta_{t})\subset \Delta$ of all systems $(\psi_{t},J_{t})$ so that the ideal restriction of $\psi_{t}$ equals $\zeta_{t}$.

Between any two objects $(\psi_{0,t},J_{0,t})$ and $(\psi_{1,t},J_{1,t})$ in $\Delta(\zeta_{t})$ there is a \emph{canonical} morphism in $\Delta$, namely, the homotopy class of an extension $\psi_{s,t}$ so that the ideal restriction of $\psi_{s,t}$ equals $\zeta_{t}$. Such an extension exists (and is unique up to homotopy through such extensions) by the Serre fibration property \S\ref{sec:fibr-sequ-sympl}.

\begin{figure}[H]
  \centering
  \begin{tikzpicture}[scale=1.5]
    \draw[line width=1pt] (0,1)--node[left]{$\psi_{0,t}$}(0,0)--node[below]{$\id$}(1,0)--node[right]{$\psi_{1,t}$}(1,1);
    \fill[pattern={Lines[angle=45]}] (0,0) rectangle (1,1);
    \path (0,1)--node[above]{$\psi_{s,1}$}+(1,0);
  \end{tikzpicture}
  \caption{A morphism between two contact-at-infinity systems is a homotopy class of extensions of $\psi_{0,t},\psi_{1,t}$ to $\psi_{s,t}$. The restriction to the top $t=1$ is required to be non-negative with respect to $s$.}
\end{figure}

Thus $\Delta(\zeta_{t})$ has an exceptionally simple category structure (exactly one morphism between any two objects).\footnote{Such a category is sometimes called an \emph{indiscrete groupoid}. The structure is also closely related to the \emph{connected simple systems} of \cite[\S{III.5.3}]{conley-book}.} We can define $\mathrm{HF}(\zeta_{t})$ as either the limit or colimit of $\mathrm{HF}(\psi_{t},J_{t})$ over $\Delta(\zeta_{t})$, and in either case $\mathrm{HF}(\zeta_{t})$ is isomorphic to any representative $\mathrm{HF}(\psi_{t},J_{t})$. This is the definition given in \cite[\S2.2.1]{cant_sh_barcode}.

One can define a category $\Delta'$ whose objects are contact systems $\zeta_{t}$ so that $\zeta_{1}$ does not have any discriminant points and whose morphisms are homotopy classes of extensions $\zeta_{s,t}$ so that $\zeta_{s,1}$ is positive, similarly to the definition of $\Delta$. By the functoriality of continuation morphisms, the assignment of $\zeta_{t}$ to $\mathrm{HF}(\zeta_{t})$ is itself a functor $\Delta'\to \mathrm{Vect}$.

To be concrete, given representatives $(\psi_{0,t},J_{0,t})$ and $(\psi_{1,t},J_{1,t})$ for $\zeta_{0,t}$ and $\zeta_{1,t}$, one lifts $\zeta_{s,t}$ to an extension $\psi_{s,t}$ using the Serre fibration property, thereby obtaining a morphism in $\Delta$. The resulting map: $$\mathrm{HF}(\psi_{0,t},J_{0,t})\to \mathrm{HF}(\psi_{1,t},J_{1,t})$$ commutes with the continuation maps in $\Delta(\zeta_{0,t})$ and $\Delta(\zeta_{1,t})$, and hence induces a map on their colimits $\mathrm{HF}(\zeta_{0,t})\to \mathrm{HF}(\zeta_{1,t})$.

Via this process we can remove the dependence on $J_{t}$ and on the precise choice of extension of $\zeta_{t}$ to $\psi_{t}$. In the rest of this paper, we will simply refer to $\mathrm{HF}(\zeta_{t})$, and sometimes $\mathrm{HF}(\psi_{t})$ when we want to emphasize a particular choice of extension of the ideal restriction $\zeta_{t}$.

It follows in a straightforward manner that the colimit of this functor $\Delta'\to \mathrm{Vect}$ is isomorphic (in a natural way) to $\mathrm{SH}(W)$, as defined in \S\ref{sec:sympl-cohom-as}.

\subsubsection{Dominating sequences}
\label{sec:domin-sequ}

We would like to compute \( \mathrm{SH}(W) \) using a smaller category than the entirety of \( \Delta', \) since this contains more objects and morphisms than are feasible to work with. We give here some criteria for this.

Let \( \mathbb{N} \) be the category with objects the natural numbers, and morphisms given by the \( \leq \) relation. Let \( \mathfrak{N}:\mathbb{N} \to \Delta' \) be a functor. Let \( \psi_{i,s,t}:\mathfrak{N}(i)\to \mathfrak{N}(i+1) \) denote (representatives) of the morphisms in the sequence. Note that \( \psi_{i,0,t} = \mathfrak{N}(i), \) and \( \psi_{i,1,t} = \mathfrak{N}(i+1). \) By convention we will suppress the parameter \( t \) when referring to the object \( \mathfrak{N}(i). \) Call \( \mathfrak{N} \) \emph{dominating} if there a $C^{\infty}$ open set $U$ of the constant paths so that $\psi_{i,s,1}\nu_{s}$ is positive for all paths $\nu_{s}$ in $U$ and for all $i$.

\begin{lemma}
  Dominating \( \mathfrak{N} \) are filtering in \( \Delta' \) and thus are cofinal.
\end{lemma}
\begin{proof}
  First we show property (i) of filtering.
  If $\mu_{t}$ is any system, then we can write $\mathfrak{N}(1)^{-1}\mu_{t}$ as a composition $\nu_{1,t}\dots\nu_{k,t}$ of systems so that: $$s\mapsto \psi_{i,s,1}\nu_{i,s}^{-1}$$ is positive for each $i=1,\dots,k$, provided $k$ is large enough; this uses the dominating condition and that any open set $U$ in the group of paths $t\mapsto \nu_{t}\in \mathrm{Cont}(Y)$ containing $\nu_{0}=\id$ generates the entire group of isotopies (a familiar property of connected topological groups; the union of $U\cup U^{2}\cup \dots $ is open and closed).

  In this way we define a family of objects \( \mathfrak{M}_{\mu}(i) := \mathfrak{N}(i) \nu_{i-1,t}^{-1}\dots\nu_{1,t} \mathfrak{N}(1)^{-1}\mu_t \), with \( \mathfrak{M}_{\mu}(1) = \mu_t \), and morphisms between them given by the squares:
  \begin{equation*}
    \xi_{i,s,t}=\psi_{i,s,t}\nu_{i,st}^{-1}\nu_{i-1,t}^{-1}\dots\nu_{1,t}^{-1}\mathfrak{N}(1)^{-1}\mu_{t} = \psi_{i,s,t}\nu_{i,st}\mathfrak{N}(i)^{-1}\mathfrak{M}_{\mu}(i),
  \end{equation*}
  for \( (s,t) \in [0,1]^{2}. \) A bit of thought reveals that: $$\xi_{i,1,t}=\xi_{i+1,0,t}=\mathfrak{M}_{\mu}(i+1),$$
  since \( \psi_{i,1,t}=\psi_{i+1,0,t}=\mathfrak{N}(i+1) \).

  Notice that $\mathfrak{M}_{\mu}(k+1)=\mathfrak{N}(k+1)$, since $\nu_{k,t}^{-1}\dots\nu_{1,t}^{-1}\mathfrak{N}(1)^{-1}\mu_{t}=\id$, by construction. The composition of the morphisms $\xi_{i,s,t}$ as $i=1,\dots,k$ is therefore a morphism from the original system $\mu_{t}$ to $\mathfrak{N}(k+1)$. It follows that property (i) holds.

  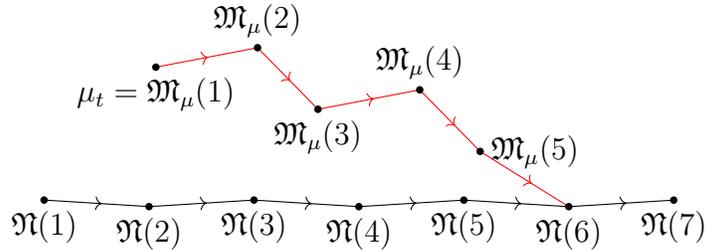
\begin{figure}[H]
    \centering
    \begin{tikzpicture}[xscale=1.4]
      \begin{scope}[rotate=-45]
        \draw[postaction={decorate,decoration={
            markings,
            mark=between positions 0.11 and 1 step 0.2 with {\arrow{>};
            }}}] (0.5,-2)coordinate(J1)--+(40:1)coordinate(J2) (J2)--+(50:1) coordinate(J3) (J3)--+(40:1)coordinate(J4) (J4)--+(50:1) coordinate(J5) (J5)--+(40:1) coordinate(J6);
        \draw[postaction={decorate,decoration={
            markings,
            mark=at position 0.55 with {\arrow{>};
            }}}] (J6)--+(50:1) coordinate(J7);

        \draw[red,postaction={decorate,decoration={
            markings,
            mark=between positions 0.11 and 1 step 0.2 with {\arrow{>};
            }}}] (0,0)coordinate(P1)--+(60:1)coordinate(P2) (P2)--+(-10:1) coordinate(P3) (P3)--+(60:1)coordinate(P4) (P4)--+(-10:1) coordinate(P5)--(J6);

        \path[every node/.style={fill,black,circle,inner sep=1pt}] (J1)node{}--(J2)node{}--(J3)node{}--(J4)node{}--(J5)node{}--(J6)node{}--(J7)node{};
        \path[every node/.style={below}] (J1)node{$\mathfrak{N}(1)$}--(J2)node{$\mathfrak{N}(2)$}--(J3)node{$\mathfrak{N}(3)$}--(J4)node{$\mathfrak{N}(4)$}--(J5)node{$\mathfrak{N}(5)$}--(J6)node{$\mathfrak{N}(6)$}--(J7)node{$\mathfrak{N}(7)$};
        \path[every node/.style={fill,black,circle,inner sep=1pt}] (P1)node{}--(P2)node{}--(P3)node{}--(P4)node{}--(P5)node{};
        \path (P1)node[below]{$\mu_{t}=\mathfrak{M}_{\mu}(1)$}--(P2)node[above]{$\mathfrak{M}_{\mu}(2)$}--(P3)node[below]{$\mathfrak{M}_{\mu}(3)$}--(P4)node[above]{$\mathfrak{M}_{\mu}(4)$}--(P5)node[right]{$\mathfrak{M}_{\mu}(5)$};
      \end{scope}
    \end{tikzpicture}
    \caption{Figure used in proof of (i) with $k=5$.}
    \label{fig:figure-property-1-dominating}
  \end{figure}

  With this established it remains to show property (ii).

  Let \( \mu_{t} \) some admissible data and, as above, let \( \psi_{1,0,t}=\mathfrak{N}(1) \). Any two morphisms from $\mu_{t}$ to $\mathfrak{N}(1)$ are represented by squares $\mu_{s,t}^{0}$ and $\mu_{s,t}^{1}$ with $\mu_{0,t}^{i}=\mu_{t}$ and $\mu_{1,t}^{i}=\mathfrak{N}(1)$, for $i=0,1$. For simplicity, we only consider morphisms to $\mathfrak{N}(1)$; the general case (two morphisms into $\mathfrak{N}(k)$) follows from the same argument.

  It is clear that there is some extension $\mu^{\eta}_{s,t}$ satisfying:
  \begin{enumerate}
  \item $\mu^{\eta}_{s,0}=\id$ for all $\eta,s$,
  \item $\mu^{\eta}_{0,t}=\mu_{t}$ and $\mu^{\eta}_{1,t}=\mathfrak{N}(1)$ for all $\eta$.
  \end{enumerate}
  The requirement that $s\mapsto \mu^{\eta}_{s,1}$ is non-negative for all $\eta$ cannot be guaranteed.

  However, if $k>0$ is sufficiently large, then we can ensure that:
  \begin{equation*}
    s\in [0,1]\mapsto \psi_{i,s,1}\mathfrak{N}(1)^{-1}\mu_{(s+i-1)/k,1}^{\eta}\text{ is positive for each $i=1,2,\dots,k$},
  \end{equation*}
  and for all $\eta$. This is because the rescaling by factor $1/k$ eventually makes $\mu^{\eta}_{(s-i-1)/k,1}$ arbitrarily close to a constant path, and the sequence is dominating.

  Consider now the family $\xi^{\eta}_{s,t}$ defined for $(\eta,s,t)\in [0,1]\times [0,k]\times [0,1]$ given by:
  \begin{equation*}
    \xi^{\eta}_{s,t}=\psi_{s-[s]+1,[s],t}\mathfrak{N}(1)^{-1}\mu_{s/k,t}^{\eta},
  \end{equation*}
  where $s\mapsto [s]$ denotes the truncation $\R\to \R/\Z$. In words, we divide $[0,k]\times [0,1]$ into $k$ squares, and on the $i$th square $\xi_{s,t}^{\eta}$ equals $\psi_{i,[s],t}\mathfrak{N}(1)^{-1}\mu^{\eta}_{s/k,t}$. The formula for $\xi_{s,t}^{\eta}$ is continuous, and can be made smooth by appropriately reparametrizing the $s$ coordinate using cut-off functions; we leave the details of this smoothing to the reader.

  One checks that:
  \begin{enumerate}
  \item $\xi_{0,t}^{\eta}=\mu^{\eta}_{0,t}=\mu_{t}$,
  \item $\xi_{k,t}^{\eta}=\mathfrak{N}(k+1)$, since $\psi_{k+1,0,t}=\mathfrak{N}(k+1)$ and $\mu^{\eta}_{1,t}=\mathfrak{N}(1)$,
  \item the restriction $\xi_{s,1}^{\eta}$ is positive for all $\eta$.
  \end{enumerate}
  Therefore, $\xi^{\eta}_{s,t}$ defines a homotopy between two representatives of morphisms $\mu_{t}\to \mathfrak{N}(k+1)$ in $\Delta$, i.e., the morphisms induced by $\xi^{0}_{s,t}$ and $\xi^{1}_{s,t}$ are the same. Note that we think of rectangles $[0,k]\times [0,1]$ as defining $k$ fold compositions of morphisms; since rectangles are considered up to homotopy when defining morphisms, we implicitly reparametrize the $s$ coordinate to have length $[0,1]$ when referring to the morphism given by $\xi^{\eta}_{s,t}$.

  Finally, consider the deformation of maps on $[-1,k]\times [0,1]$ given by:
  \begin{equation*}
    \xi^{0,\tau}_{s,t}=\left\{
      \begin{aligned}
        &\mu^{0}_{\tau\beta(s+1),t}&-1\le s\le 0&\\
        &\psi_{s-[s]+1,[s],t}\mathfrak{N}(1)^{-1}\mu^{0}_{(1-\tau)s/k+\tau,t}&0\le s\le k&,
      \end{aligned}
    \right.
  \end{equation*}
  where $\beta:\R\to [0,1]$ is an increasing cut-off function so $\beta(x)=0$ for $x\le 0$ and $\beta(x)=1$ for $x\ge 1$. The restriction to the $s$-axis is non-negative for each $\tau$, since $\mu^{0}_{s,t}$ is presumed to be non-negative. One checks that $\xi_{s,t}^{0,\tau}$ is continuous, and, when $\tau=1$ we have:
  \begin{equation*}
    \xi^{0,1}_{s,t}=\left\{
      \begin{aligned}
        &\mu^{0}_{\beta(s+1),t}&-1\le s\le 0&\\
        &\psi_{s-[s]+1,[s],t}&0\le s\le k&,
      \end{aligned}
    \right.
  \end{equation*}
  This equals to the $k+1$-fold composition of morphisms, from $\mu_{t}$ to $\mathfrak{N}(1)$ via $\mu^{0}_{s,t}$, then to $\mathfrak{N}(2)$, etc, until $\mathfrak{N}(k+1)$.

  A similar argument constructs $\xi^{1,\tau}_{s,t}$, so that $\xi^{1,\tau}_{s,t}=\xi^{1}_{s,t}$ and so $\xi^{1,1}_{s,t}$ equals a similar $k+1$-fold composition of morphisms, starting instead with $\mu^{1}_{s,t}$ from $\mu_{t}$ to $\mathfrak{N}(1)$.

  Since $\xi^{1}_{s,t}$ and $\xi^{0}_{s,t}$ define the same morphism (as shown above), it follows that the two $k+1$-fold compositions are equal, which is exactly the filtering property (ii). This completes the proof.
\end{proof}

One obvious dominating sequence is simply $R^{\alpha}_{x_{n}t}$ where $x_{n}$ is increasing with a minimal gap $x_{n+1}-x_{n}>\epsilon$, and so that $x_{n}$ is never a period of a closed $\alpha$-Reeb orbit. The morphisms are given by the interpolation $\xi_{s,t}=R_{(1-s)x_{n}t+s x_{n+1} t}^{\alpha}$.

Similarly, the sequence with \( \mathfrak{N}(1)=\psi_{t} \) any non-degenerate contactomorphism and all subsequent objects (and morphisms) given by composing with a strictly positive loop of contactomorphisms $\varphi_{t}$ is also dominating; such a sequence is considered in \S\ref{sec:the-twist-trick}.

\subsection{PSS and the unit in symplectic cohomology}
\label{sec:PSS-and-the-unit-sympl-cohom}
Let $R_{s}^{\alpha}$ be the time $s$ Reeb flow. Autonomous systems whose ideal restriction is $R_{s}^{\alpha}$ are the systems considered in \cite{ritter_TQFT}, (although with different notation).

In this section we describe the pair-of-pants product and the unit for symplectic cohomology following \cite[\S6]{ritter_TQFT}; see also \cite{schwarz-thesis,salamon99-quantum-products,seidel-eq-pop,alizadeh-atallah-cant}. The product is defined by counting solutions to Floer's equation modeled on a pair of pants surface (thrice punctured sphere) with two ends marked as inputs (positive) and one as an output (negative). Such a count defines, for slopes \( a,b > 0 \), a homomorphism:
\[ \mathrm{HF}(R^{\alpha}_{at}) \otimes \mathrm{HF}(R^{\alpha}_{bt}) \to \mathrm{HF}(R^{\alpha}_{(a+b)t}). \]
Taking colimits yields a map \( \mathrm{SH}(W) \otimes \mathrm{SH}(W) \rightarrow \mathrm{SH}(W) \), which endows \( \mathrm{SH}(W) \) with the structure of a unital ring; see \cite[Theorems 6.1, A.10, A.12, A.14]{ritter_TQFT}.

The same construction also gives \( \mathrm{HF}(R^{\alpha}_{\epsilon t}) \) the structure of a unital ring when $\epsilon$ is sufficiently small -- one uses that for $\epsilon$ sufficiently small there is a continuation isomorphism \( \mathrm{HF}(R^{\alpha}_{\epsilon t}) \simeq \mathrm{HF}(R_{2\epsilon t}^{\alpha}) \). In particular, the colimit map is a unital ring homomorphism:
\[ c:\mathrm{HF}(R^{\alpha}_{\epsilon t}) \rightarrow \mathrm{SH}(W). \]
In \cite[\S6.8, \S6.9, \S15]{ritter_TQFT} it is shown that the PSS map \( H^*(W) \rightarrow \mathrm{HF}(R^{\alpha}_{\epsilon t}) \) is a ring isomorphism where $H^*(W)$ carries the usual cup product and $\epsilon$ is sufficiently small; see \S\ref{sec:pss-comparison} further discussion of the PSS map.

Thus the induced map \( H^*(W) \rightarrow \mathrm{SH}(W) \) is a unital ring homomorphism. In particular, the unit in symplectic cohomology is the image of the unit of ordinary cohomology (\cite[\S6.8]{ritter_TQFT}).

The above can be summarized by the following:
\begin{theorem}[{\cite[Theorems 6.1, 6.6, A.14]{ritter_TQFT}}]\label{thm:sh-ring-str}
For $W$ a Liouville manifold, the symplectic cohomology, \( \mathrm{SH}(W), \) with the ``pair of pants product'' is a graded-commutative ring with unit. Moreover there is a (unital) ring homomorphism:
\begin{equation}\label{eq:SHring-hom}
  H^*(W) \cong \mathrm{HF}(R_{\epsilon t}^{\alpha}) \rightarrow \mathrm{SH}(W),
\end{equation}
where \( H^*(W) \) is equipped with the cup product.
\end{theorem}

\subsubsection{Subcomplex generated by contractible orbits}
\label{sec:subc-gener-contr}

Technically \cite{ritter_TQFT} works in the context of Liouville manifolds. However, it is straightforward to show that one has the required energy bounds (without the need to use Novikov coefficients) when defining the product:
\begin{equation*}
  \mu:\mathrm{HF}_{0}(R^{\alpha}_{a t})\otimes \mathrm{HF}_{0}(R^{\alpha}_{b t})\to \mathrm{HF}_{0}(R^{\alpha}_{(a+b)t}),
\end{equation*}
for $a,b>0$ where $\mathrm{HF}_{0}$ denotes the subcomplex spanned by the contractible orbits. One uses fixed cappings of contractible orbits to obtain a priori bounds on the symplectic areas of pairs-of-pants. One only needs the aspherical condition for this a priori energy bound to hold.

In particular, we conclude the following corollary which is all that we will need from the pair-of-pants product:
\begin{cor}\label{cor:aspherical-pop}
  Let $W$ be a symplectically aspherical and convex-at-infinity manifold. Then $\mathrm{SH}_{0}(W)$ is a unital ring, and the continuation map: $$H^{*}(W)\simeq \mathrm{HF}_{0}(R^{\alpha}_{\epsilon t})\to \mathrm{SH}_{0}(W)$$ is a ring homomorphism provided $\epsilon>0$ is smaller than the minimal period of a closed $\alpha$-Reeb orbit.
\end{cor}
In this case we should mention that $\mathrm{HF}_{0}(R^{\alpha}_{\epsilon t})\simeq \mathrm{HF}(R^{\alpha}_{\epsilon t})$ if $\epsilon$ is small enough.

\subsubsection{PSS maps and comparison with Morse cohomology}
\label{sec:pss-comparison}
Let $f$ be a Morse function on $W$ which is non-vanishing and $1$-homogeneous in the convex end (so all the critical points lie in a compact set). As in Floer cohomology, one defines Morse cohomology $\mathrm{HM}^{*}(f)$ by counting negative gradient flow lines whose input is the asymptotic at $s=+\infty$.

\begin{theorem}\label{theorem:pss-factor}
  Let $f_{\pm}$ be $1$-homogenous and positive, resp., negative, in the convex end. There is a commutative diagram:
  \begin{equation*}
    \begin{tikzcd}
      {\mathrm{HM}^{*}(f_{-})}\arrow[r,"{}"] &{\mathrm{HM}^{*}(f_{+})}\arrow[d,"{}"]\\
      {\mathrm{HF}(R^{\alpha}_{-\epsilon t})}\arrow[u,"{}"]\arrow[r,"{}"] &{\mathrm{HF}(R^{\alpha}_{\epsilon t})},
    \end{tikzcd}
  \end{equation*}
  where the vertical maps are the PSS morphisms, as in \cite{ritter_TQFT,frauenfelder_schlenk}, and the horizontal maps are induced by continuation maps.
\end{theorem}
See \cite[Proposition 1.3]{cieliebak_frauenfelder_oancea} for a related statement. Note that the left hand vertical map in the above diagram goes in the \textit{opposite} direction to the usual PSS map. This is because we work with negative slope, and must continue to slope 0 when defining the PSS map. That the above diagram commutes follows from a breaking/gluing analysis analogous to that involved in the proof that the PSS map in the closed case is an isomorphism as in \cite{pss}, and the usual compatibility with continuation; see Figure \ref{fig:PSS}. See also \cite{biran-cornea,biran_cornea_CRM,biran-cornea-rigidity-uniruling,biran_cornea_lagrangian_topology} for similar arguments considering configurations involving Morse flow lines in Floer theory.

As explained in \S\ref{sec:PSS-and-the-unit-sympl-cohom}, the vertical map on the right hand side of the above diagram is an isomorphism if $\epsilon$ is smaller than the minimal period of a closed Reeb orbit; see \cite{ritter_TQFT,frauenfelder_schlenk} for the proof. The same follows on the left hand side, by carrying out the same procedure, just reversing the spiked disks, holomorphic curves, and flow lines at all stages (taking negative slope to positive slope).

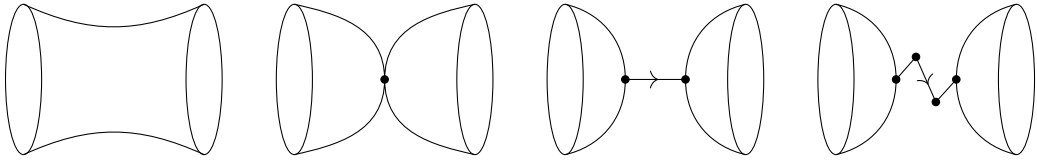
\begin{figure}[H]
  \centering
  \begin{tikzpicture}[xscale=0.8]
    \begin{scope}
      \draw (0,0) circle (0.3 and 1) coordinate[shift={(0,1)}](A) coordinate[shift={(0,-1)}](B);
      \draw (3,0) circle (0.3 and 1) coordinate[shift={(0,1)}](C) coordinate[shift={(0,-1)}](D);
      \draw (A)to[out=-20,in=200](C) (B)to[out=20,in=160](D);
    \end{scope}
    \begin{scope}[shift={(4.5,0)}]
      \draw (0,0) circle (0.3 and 1) coordinate[shift={(0,1)}](A) coordinate[shift={(0,-1)}](B);
      \draw (3,0) circle (0.3 and 1) coordinate[shift={(0,1)}](C) coordinate[shift={(0,-1)}](D);
      \coordinate (X) at (1.5,0);
      \draw (A)to[out=-10,in=90](X)to[out=-90,in=10](B) (C)to[out=190,in=90](X)to[out=-90,in=170](D);
      \node[draw,circle,inner sep=1pt,fill] at (X){};
    \end{scope}
    \begin{scope}[shift={(9,0)}]
      \draw (0,0) circle (0.3 and 1) coordinate[shift={(0,1)}](A) coordinate[shift={(0,-1)}](B);
      \draw (3,0) circle (0.3 and 1) coordinate[shift={(0,1)}](C) coordinate[shift={(0,-1)}](D);
      \path (1.0,0) coordinate (X) -- (2,0) coordinate (Y);
      \draw (A)to[out=-10,in=90](X)to[out=-90,in=10](B) (C)to[out=190,in=90](Y)to[out=-90,in=170](D);
      \draw[every node/.style={draw,circle,inner sep=1pt,fill},postaction={decorate,decoration={
          markings,
          mark=at position 0.55 with {\arrow{>[scale=1.3]};},
        }}] (X)node{}--(Y)node{};
    \end{scope}
    \begin{scope}[shift={(13.5,0)}]
      \draw (0,0) circle (0.3 and 1) coordinate[shift={(0,1)}](A) coordinate[shift={(0,-1)}](B);
      \draw (3,0) circle (0.3 and 1) coordinate[shift={(0,1)}](C) coordinate[shift={(0,-1)}](D);
      \path (1.0,0) coordinate (X) -- (1.33,0.3) coordinate (AA) -- (1.66,-0.3) coordinate(BB) -- (2,0)coordinate(Y);
      \draw (A)to[out=-10,in=90](X)to[out=-90,in=10](B) (C)to[out=190,in=90](Y)to[out=-90,in=170](D);
      \draw[every node/.style={draw,circle,inner sep=1pt,fill},postaction={decorate,decoration={
          markings,
          mark=at position 0.55 with {\arrow{>[scale=1.3]};},
        }}] (X)node{}--(AA)node{}--(BB)node{}--(Y)node{};
    \end{scope}
  \end{tikzpicture}
  \caption{The moduli spaces used to prove Theorem \ref{theorem:pss-factor}. These induce morphism $\mathrm{HF}(R_{-\epsilon t})\to \mathrm{HF}(R_{\epsilon t})$ by considering the positive puncture as the input.}
  \label{fig:PSS}
\end{figure}
\begin{figure}[H]
  \centering
  \begin{tikzpicture}[xscale=1.7]
    \draw (0,0) circle (0.1 and 1) coordinate[shift={(0,1)}](A) coordinate[shift={(0,-1)}](B) coordinate[shift={(-0.1,0)}](PS);
    \draw (4,0) circle (0.1 and 1) coordinate[shift={(0,1)}](C) coordinate[shift={(0,-1)}](D) coordinate[shift={(0.1,0)}](NS);
    \path (0.8,0) coordinate (X) -- (3.2,0) coordinate (Y);
    \draw (A)to[out=0,in=90](X)to[out=-90,in=0](B) (C)to[out=180,in=90](Y)to[out=-90,in=180](D);

    \draw[every node/.style={draw,circle,inner sep=1pt,fill},postaction={decorate,decoration={
        markings,
        mark=at position 0.55 with {\arrow{>[scale=1.3]};},
      }}] (X)node{}--coordinate[shift={(0,0.1)}](FL)(Y)node{};
    \path
    (PS)node[left]{PSS for $R_{+\epsilon}^{\alpha}$}--(FL)node[above]{$x'(s)=\xi_{s}(x(s))$}--(NS)node[right]{PSS for $R_{-\epsilon}^{\alpha}$};
  \end{tikzpicture}
  \caption{More detailed view of the third piece of Figure \ref{fig:PSS}.}
  \label{fig:PSS-more-detail}
\end{figure}
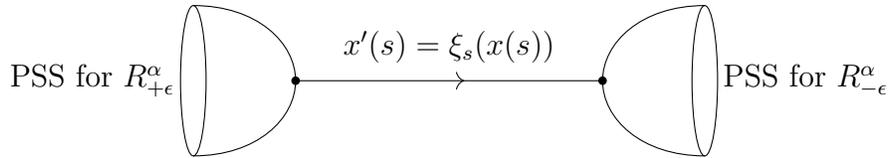

One subtlety which should be mentioned is that the flow line appearing in the third part of Figure \ref{fig:PSS} is a \emph{continuation line} interpolating from $-\nabla f_{+}$ (on the left) to $-\nabla f_{-}$ (on the right). The fourth piece in Figure \ref{fig:PSS} is a breaking interpreted as the composition: $$\mathrm{HF}(R^{\alpha}_{-\epsilon t})\to \mathrm{HM}(f_{-})\to \mathrm{HM}(f_{+})\to \mathrm{HF}(R^{\alpha}_{\epsilon t}),$$ going from right to left as per the cohomological conventions used in this paper.

More precisely, if: $$\xi_{s}=-\nabla f_{+}+\beta(s)(\nabla f_{+}-\nabla f_{-}),$$ the third part of Figure \ref{fig:PSS} has a flow line for $\xi_{s}$ over the region $s\in [-\ell,\ell]$. Here $\ell$ is considered as a parameter, so that $\ell=0$ describes the configurations shown in the second part of Figure \ref{fig:PSS}, and $\ell\to\infty$ describes the fourth part of Figure \ref{fig:PSS}; see Figure \ref{fig:PSS-more-detail} for more details.

We should remark that any flow line for $\xi_{s}$ over $[-\ell,\ell]$ whose endpoints remains in a compact set $K$ will remain in some other compact set $K'$, independently of how large $\ell$ is. This is not the case for flow lines of, e.g., $-\xi_{s}$.

\subsubsection{When is the unit born?}
\label{sec:when-unit-born}

As a consequence of Theorem \ref{theorem:pss-factor} and the discussion in \S\ref{sec:pss-comparison}, the image of the continuation map $\mathrm{HF}(R_{-\epsilon t}^{\alpha})\to \mathrm{HF}(R_{\epsilon t}^{\alpha})$ consists entirely of nilpotent elements with respect to the pair-of-pants product. Indeed, this holds because $\mathrm{HM}(f_{+})\to \mathrm{HF}(R_{\epsilon t}^{\alpha})$ is a ring homomorphism and the image of $\mathrm{HM}(f_{-})\to \mathrm{HM}(f_{+})$ consists of nilpotent elements, by degree considerations.

In particular, if $\mathrm{SH}_{0}(W)\ne 0$, then the map $\mathrm{HF}(R^{\alpha}_{-\epsilon t})\to \mathrm{SH}_{0}(W)$ does \emph{not} hit $1$ (because $1\ne 0$ is not nilpotent). One implicitly applies Corollary \ref{cor:aspherical-pop}.

Referring to the barcode associated to the persistence module $V_{s}=\mathrm{HF}(R^{\alpha}_{st})$, there is a half-infinite bar corresponding to $1$ which is born at parameter $s=0$; see \cite[\S2.5]{cant_sh_barcode} for related discussion.

This observation that $\mathrm{HF}(R^{\alpha}_{-\epsilon t})\to \mathrm{SH}(W)$ is not surjective when the unit is nonzero is crucial in the proof of Theorem \ref{theorem:main-absolute}. Indeed, we have:
\begin{lemma}\label{lemma:non-surjective-lemma}
  The continuation morphism $\mathrm{HF}(R^{\alpha}_{-\epsilon t})\to \mathrm{SH}(W)$ is not surjective if $\mathrm{SH}(W)\ne 0$.
\end{lemma}
\begin{proof}
  If $\mathrm{SH}_{0}(W)\ne \mathrm{SH}(W)$, then $\mathrm{HF}(R^{\alpha}_{-\epsilon t})\to \mathrm{SH}(W)$ is not surjective by consideration of the free homotopy classes of orbits (since $\mathrm{HF}(R^{\alpha}_{-\epsilon t})=\mathrm{HF}_{0}(R^{\alpha}_{-\epsilon t})$ for $\epsilon$ small enough). If $\mathrm{SH}_{0}(W)=\mathrm{SH}(W)$, then we can use the fact that $\mathrm{SH}_{0}(W)\ne 0$ if and only if the unit is non-zero. The preceding ``nilpotency'' argument shows that a non-zero unit is never in the image of $\mathrm{HF}(R^{\alpha}_{-\epsilon t})\to \mathrm{SH}(W)$. This completes the proof.
\end{proof}

\subsection{The twisting trick}
\label{sec:the-twist-trick}

We follow the strategy of \cite{uljarevic_floer_homology_domains,merry_ulja}; see also \cite{ritter_negative_line_bundles} and \cite[Theorem 2.3]{ritter_circle_actions}.

Let $\varphi_{t}$ be a loop in $\mathrm{Ham}(W)$ based at $1$ whose ideal restriction is positive (i.e., an extensible positive loop). Fix $\psi_{t}$ to be any admissible contact-at-infinity system.

Define:
\begin{equation*}
  \psi^{k}_{t}:=\varphi_{t}\circ \dots \circ \varphi_{t}\circ \psi_{t},
\end{equation*}
and also define $\mathfrak{c}_{k}$ to be the continuation map $\mathrm{HF}(\psi^{k}_{t})\to \mathrm{HF}(\psi^{k+1}_{t})$ given by the path $\mu_{s,t}=\varphi_{st}\psi^{k}_{t}$. As discussed in \ref{sec:domin-sequ}, the sequence $\psi^{k}_{t}$ is dominating, and thus cofinal. Using this sequence $\psi^{k}_{t}$, we will show that $\mathrm{SH}(W)=0$.

Applying the naturality isomorphism $\mathfrak{n}:\mathrm{HF}(\psi_{t}^{0})\to \mathrm{HF}(\psi_{t}^{k})$, one concludes that the all of the vector spaces $\mathrm{HF}(\psi^{k}_{t})$ have the same dimension, say $d$. It follows easily that $\dim \mathrm{SH}(W)\le d$. Moreover, because naturality maps commute with continuation maps (as in \S\ref{sec:natur-transf}), the following diagram is commutative:
\begin{equation*}\begin{tikzcd}
    {\mathrm{HF}(\psi_{t}^{0})}\arrow[d,"{\mathfrak{n}}"]\arrow[r,"{\mathfrak{c}}"] &{\mathrm{SH}(W)}\arrow[d,"{\id}"]\\
    {\mathrm{HF}(\psi_{t}^{k})}\arrow[r,"{\mathfrak{c}}"] &{\mathrm{SH}(W)}.
  \end{tikzcd}
\end{equation*}
By taking $k$ sufficiently large, we can ensure that the lower continuation map is surjective (becuase the dimension of $\mathrm{SH}(W)$ is finite), and thereby conclude that the upper continuation map is also surjective, as desired. However, if $\mathrm{SH}(W)\ne 0$, then we have seen in Lemma \ref{lemma:non-surjective-lemma} that, if $\psi^{0}_{t}=R^{\alpha}_{-\epsilon t}$, then the natural map $\mathrm{HF}(\psi^{0}_{t})\to \mathrm{SH}(W)$ is \emph{not} surjective. Thus we conclude Theorem~\ref{theorem:main-absolute}.

\subsection{Stabilizations and non-negative extensible loops}
\label{sec:stab-liouv-manif}
Every stabilization admits a non-negative extensible loop:
\begin{lemma}
  The loop of Hamiltonian diffeomorphisms:
  \begin{equation*}
    \Phi_t: W'\rightarrow W'\text{ given by }(p,z)\mapsto (p, e^{2\pi i t}z).
  \end{equation*}
  is contact-at-infinity and has a non-negative ideal restriction.
\end{lemma}
\begin{proof}
  The loop $\Phi_t$ is generated by the Hamiltonian $H(p,z)=\pi|z|^2$, which is $1$-homogeneous with respect to the Liouville flow. Since the generating Hamiltonian is non-negative at infinity, the desired result follows.
\end{proof}

\subsection{The ergodic trick of Eliashberg-Polterovich}
\label{sec:ergod-trick-eliashb}

A special property of the group of contactomorphisms and its universal cover is that the existence of a non-negative loop implies the existence of a strictly positive loop; see \cite[Proposition 2.1.A]{ep2000}. Moreover, if the original non-negative loop is contractible, then the resulting positive loop will also be contractible. In this subsection we show that a similar property holds for extensible loops.
\begin{lemma}\label{lem:ergodic}
    Let $\phi_t$ be an extensible non-negative loop of contactomorphisms of the ideal boundary $Y$ of a Liouville manifold $(W,\lambda)$. Then there is an extensible positive loop of contactomorphisms based at the identity.
\end{lemma}
\begin{proof}
The Serre fibration property for the ideal restriction implies that $\psi_{t}$ being an extensible loop depends only on the free homotopy class $[\psi_{t}]$ in $\mathrm{Cont}(Y,\xi)$. Moreover, it suffices prove the existence of an extensible positive loop $\psi_{t}$ based at any point in $\mathrm{Cont}_{0}(Y)$, as $\psi_{t}\psi_{0}^{-1}$ will be extensible and based at the identity.

Let $\phi_t$ be an extensible non-negative loop of contactomorphisms. Using the construction from the proof of \cite[Proposition 2.1.A]{ep2000} one obtains a positive loop $\psi_t$ as a composition of conjugates of time shifts of the original loop $\phi_{t}$, where we conjugate by elements in $\mathrm{Cont}_{0}(Y,\xi)$.

From this construction, one observes that the free homotopy classes of the loops are related by $[\psi_t]=[\phi_t^k]$ for some positive integer $k$ (since the conjugation of a loop by an element of $\mathrm{Cont}_{0}(Y,\xi)$ does not change the homotopy class). Since $\phi^{k}_{t}$ is extensible, $\psi_{t}$ is also extensible, and hence we have furnished an extensible positive loop, as desired.
\end{proof}

\subsection{Exotic symplectomorphisms and the flux}
\label{sec:exot-sympl-flux}

The goal is to prove Theorem~\ref{theorem:exotic-mapping-class}. Thus, let $\psi_{t}$ be an isotopy in $\mathrm{Ham}(W)$ based at the identity whose ideal restriction is a positive loop. Then $[\psi_{1}]$ represents a mapping class which is known to be non-trivial in $\pi_{0}(\Ham_{c}(W))$, because of Theorem \ref{theorem:main-absolute}; it remains to explain why it is non-trivial in $\pi_{0}(\mathrm{Symp}_{c}(W))$.

Arguing by contradiction, suppose that $[\psi_{1}]$ is trivial in $\pi_{0}(\mathrm{Symp}_{c}(W));$ then there exists some compactly supported symplectic isotopy $\mu_{t}$ so that $\mu_{1}=\psi_{1}$; as usual we assume that $\mu_{t+1}=\mu_{t}\mu_{1}$.

The key idea is that $\mu_{t}$ has vanishing flux because $W$ is atoroidal. This in turn implies that there is a compactly supported isotopy $\eta_{t}$ so that:
\begin{enumerate}
\item $\eta_{1}=\mu_{1}$ (same endpoints),
\item $\eta_{t}$ is Hamiltonian.
\end{enumerate}

Once we establish the existence of such a deformation, we obtain the desired contradiction, since then $\eta_{t}^{-1}\psi_{t}$ is a loop in $\mathrm{Ham}(W)$ whose ideal restriction is a positive loop (contradicting Theorem \ref{theorem:main-absolute} since we assume $\mathrm{SH}(W)\ne 0$).

\subsubsection{Vanishing flux and the existence of a deformation through loops}
\label{sec:vanish-flux-exist}

We show the existence of a deformation $\eta_{t}$ satisfying (i), (ii).

Define the \emph{flux one-form} via the formula:
\begin{equation*}
  F_{\tau}:=F(\mu_{t};\tau)=\int_{0}^{\tau}\mu_{t}^{*}[\omega(-,X_{t})]\d t,
\end{equation*}
where $X_{t}$ is the generator of $\mu_{t}$. This is a closed, compactly supported one-form of $W$. It is well-known that, if $\Gamma$ is a loop, then:
\begin{equation*}
  \int_{\Gamma}F_{\tau}=\pm \int_{\Gamma\times [0,1]}\omega,
\end{equation*}
where $\Gamma\times [0,1]$ is shorthand for the cylinder obtained by flowing $\Gamma$ by $\mu_{t}$. In particular, since $W$ is atoroidal, $F_{1}$ is an exact one-form. Indeed, the symplectic area of $(s,t)\mapsto \mu_{t}(\Gamma(s))$ equals the symplectic area of $(s,t)\mapsto \psi_{t}(\Gamma(s))$ since $\psi_{1}=\mu_{1}$; the latter symplectic area vanishes because $\psi_{t}$ is Hamiltonian.

The results of \cite[\S10.2]{mcduffsalamon-alt} then imply that there is a compactly supported Hamiltonian isotopy $\eta_{t}$ so that $\eta_{1}=\mu_{1}$, as desired. This completes the proof of Theorem~\ref{theorem:exotic-mapping-class}.

\subsection{Some analysis of the magnetic two-sphere}
\label{sec:some-analys-magn}

We prove Lemmas \ref{lemma:non-deg-magnetic} and \ref{lemma:magnetic-circle-action}.

\subsubsection{Non-degeneracy of the magnetic symplectic structure}
\label{sec:non-degen-magn}

This section is concerned with the proof of Lemma~\ref{lemma:non-deg-magnetic}. It is clear that $\Omega$, as defined in \eqref{eq:magnetic-symplectic-form}, extends smoothly to the zero section, as $f'(r)=0$ for $r$ small enough. Moreover, $\Omega=\d\lambda$ for $r$ large enough. It remains only to show that $\Omega$ is non-degenerate; one computes:
\begin{equation}
  \Omega=\d\lambda+f(r)\d\Pi+f'(r)\d r\wedge \Pi,
\end{equation}
and thus:
\begin{equation}\label{eq:formula-for-wedge}
  \Omega\wedge \Omega=\d\lambda\wedge \d\lambda+2f(r)f'(r)\d r\wedge \Pi\wedge \d\Pi;
\end{equation}
most of the diagonal and cross terms vanish; one slightly subtle part uses that $\Pi\wedge \d\lambda$ and $\lambda\wedge \d\Pi$ vanish pointwise when restricted to $\Sigma(r)$. Indeed, since their difference is exact and both are $\mathrm{SO}(3)$-invariant, it is sufficient to prove $\lambda\wedge \d\Pi$ vanishes, which holds since $\d\Pi$ is proportional to $\d q_{1}\wedge \d q_{2}$ in local canonical coordinates.

To prove \eqref{eq:formula-for-wedge} is nowhere zero, it is convenient to introduce the exponential coordinate $r=e^{s}$ so that $\lambda=e^{s}\alpha$. Then:
\begin{equation*}
  \Omega\wedge \Omega=2\d r\wedge (e^{s}\alpha\wedge \d\alpha+f(r)f'(r)\Pi\wedge \d\Pi).
\end{equation*}
This is non-degenerate if $\delta$ is small enough, as desired.

\subsubsection{A Hamiltonian $S^1$ action in the presence of a magentic field}
\label{sec:hamilt-s1-acti}
The goal is to prove Lemma~\ref{lemma:magnetic-circle-action}. As in the proof that $\Omega$ is non-degenerate, it is convenient to introduce exponential coordinates $r=e^{s}$, and $\lambda=e^{s}\alpha$.

The restriction of $\Omega$ to $\Sigma(r)$ equals:
\begin{equation*}
  \d(e^{s}\alpha+f(e^{s})\Pi)=e^{s}\d\alpha+f(e^{s})\d\Pi.
\end{equation*}
To analyze the kernel of this, we introduce two vector fields: let $R$ and $Y$ be the unique $\mathrm{SO}(3)$-equivariant vector fields on $\Sigma(1)$ spanning the kernels of $\d\alpha$ and $\d\Pi$, respectively, and satisfying $\Pi(Y)=1$ and $\alpha(R)=1$. Both of these are in fact Reeb vector fields for different contact structures on $\Sigma(1)$. Since $\d\alpha(Y,-)$ and $\d\Pi(R,-)$ are proportional (both vanish on the span of $Y$ and $R$), it follows there is a non-zero constant $A$ so that $A\d\alpha(Y,-)=\d\Pi(R,-)$. One can presumably find the value of $A$, however we will not require this for our argument.

Extend $R,Y$ to the complement of the zero section by fiberwise radial projection. Then the vector field:
\begin{equation}\label{eq:unnormalized_X}
  X=-Af(e^{s})Y+e^{s}R
\end{equation}
spans the characteristic foliation of each $\Sigma(r)$ hypersurface. Moreover, the restriction of $X$ to each hypersurface is $\mathrm{SO}(3)$-equivariant, and hence (by knowledge of the $\mathrm{SO}(3)$-equivariant vector fields on $\mathrm{SO}(3)$), the flow by $X$ is periodic on each hypersurface (with varying periods).

First we show that $X$ is Hamiltonian, and then we explain how to renormalize its flow so as to define a circle action. Since $X$ spans the characteristic foliation of each hypersurface, we have that:
\begin{equation*}
  \Omega(-,X)=\Omega(\bd_{s},X)\d s=(e^{s}\alpha(X)+e^{s}f'(e^{s})\Pi(X))\d s.
\end{equation*}
It is clear that $\alpha(Y)=0$ (since $Y$ points along the cotangent fibers) and $\Pi(R)=0$, since $\Pi\wedge \d\alpha=0$ (as explained in the proof of Lemma \ref{lemma:non-deg-magnetic}) and $R$ spans the kernel of $\d\alpha$. Thus:
\begin{equation*}
  2\Omega(-,X)=2(e^{2s}-Ae^{s}f'(e^{s})f(e^{s}))\d s=\d (e^{2s}-Af(e^{s})^{2}),
\end{equation*}
and so $H=e^{2s}-Af(e^{s})^{2}$ satisfies $X_{H}=2X$.

Assume that $\delta$ is small enough that $e^{2s}-Af(e^{s})^{2}$ has a uniformly positive derivative, and let $\mathfrak{p}(h)>0$ be the period of the flow by $X_{H}$ on the level set $\set{H=h}$. From the formula for $X$ in \eqref{eq:unnormalized_X} it is clear that $\mathfrak{p}(h)$ is bounded from below (unlike the non-magnetic case).

The minimum value of $H$ is $-A$, and $\mathfrak{p}(h)$ extends smoothly to $[-A,\infty)$, because the formula for $X$ extends smoothly to a vector field on the compactification $[0,\infty)\times \Sigma(1)$ obtained by setting $r=e^{s}$.

Let $k:[-A,\infty)\to \R$ be an antiderivative for $1/\mathfrak{p}(h)$, and then let
\begin{equation*}
  K=k(H)\implies X_{K}=k'(H)X_{H},
\end{equation*}
so that $X_{K}$ is $1$-periodic and Hamiltonian. Moreover, since $H$ and $k$ are smooth, so is $K$. For this we needed to know that $1/\mathfrak{p}(h)$ was smooth near $h=-A$, a fact which does not hold without the magnetic contribution.

Finally note that $X_{K}$ agrees with $R$ outside of a compact set, completing the proof of Lemma~\ref{lemma:magnetic-circle-action}.

%%% \S3
\section{Wrapped Floer cohomology and positive loops of Legendrians}
\label{sec:wrapped-floer-and-positive-loops}

In this section we develop the open string theory; in \S\ref{sec:natur-transf-assoc} we complete the proof of Theorem \ref{theorem:main-relative} and in \S\ref{sec:lagrangian-flux} we analyze the Lagrangian analogue of the flux and prove Theorem \ref{theorem:lagrangian-exotic}.

\subsection{Serre fibration property for the Lagrangian ideal restriction}
\label{sec:lagrangian-ideal-restriction}

As explained in \S\ref{sec:exot-fillings-lagr}, there is a Serre fibration $\mathrm{Lag}(W)\to \mathrm{Leg}(Y)$, where:
\begin{enumerate}
\item $\mathrm{Lag}(W)$ is the space of Lagrangians which are contact-at-infinity, i.e., which are tangent to the Liouville vector field at infinity,
\item $\mathrm{Leg}(Y)$ is the space of Legendrians in $Y$;
\end{enumerate}
see \S\ref{sec:space-cylindrical-lags} for the topology on these spaces and \S\ref{sec:serre-fibr-prop} for the proof of the Serre fibration property. We denote by $\mathrm{Lag}(W;\Lambda)$ the fiber over $\Lambda$.

The ideal restriction morphism is straightforward to define: intersect $L$ with the complement of a large starshaped domain and then project to the ideal boundary. Our assumptions ensure that the ideal restriction is a Legendrian.

\subsubsection{The topology on the space of contact-at-infinity Lagrangians}
\label{sec:space-cylindrical-lags}
The statement that the above sequence is fibrant necessarily involves some topology on the sets of Legendrians and contact-at-infinity Lagrangians. We will now describe these spaces precisely.

For a closed contact manifold \( (Y,\ker(\xi)) \), the space of Legendrians \( \mathrm{Leg}(Y) \) carries the \( C^{\infty} \) topology (i.e., the topology induced by Legendrian isotopy).

When \( W \) is a convex-at-infinity symplectic manifold, we topologize \( \mathrm{Lag}(W) \) with the $C^{\infty}$ topology; the precise definition is similar to the one in \S\ref{sec:topol-group-cont}.

It is also important to consider $\mathrm{Lag}^{\mathrm{x}}(W)$ which is the same underlying set but with the more restrictive \emph{Hamiltonian topology}.
The Hamiltonian topology on the space of Lagrangians is such that every continuous map from a disk $\mathfrak{L}:D\to \mathrm{Lag}\mathrm{x}(W)$ extends to $\varphi:D\to \mathrm{Ham}(W)$ in such a way that $\varphi_{x}(\mathfrak{L}(0))=\mathfrak{L}(x)$.
The identity map $\mathrm{Lag}^{\mathrm{x}}(W)\to \mathrm{Lag}(W)$ is continuous (but not vice-versa); moreover, the connected component of $L$ in $\mathrm{Lag}^{\mathrm{x}}(W)$ only contains those Lagrangians which are isotopic to $L$ via a contact-at-infinity system.
We refer the reader to \cite[Remark 1.1]{seidel-lecture-4d-dehn-twist} for related discussion; the space $\mathrm{Lag}^{\mathrm{x}}(W)$ is similar to the space $\mathscr{L}$ considered in \cite{chekanov_2000}.

Note that the Hamiltonian topology agrees with the \( C^{\infty} \) topology whenever the underlying submanifold has vanishing first cohomology, relative its end. One obstruction to realizing a Lagrangian isotopy by an ambient Hamiltonian isotopy is the class in \( H^1(L,\partial L) \) induced by the time derivative of the isotopy (this is a relative version of the flux); see \S\ref{sec:lagrangian-flux} for further discussion.

If $\mathrm{Lag}^{\mathrm{x}}(W)\to \mathrm{Leg}(Y)$ is a Serre fibration, then, since $\mathrm{Lag}^{\mathrm{x}}(W)\to \mathrm{Lag}(W)$ is continuous, the map $\mathrm{Lag}(W)\to \mathrm{Leg}(Y)$ will also be a Serre fibration. Thus we will focus on proving that $\mathrm{Lag}^{\mathrm{x}}(W)\to \mathrm{Leg}(Y)$ is a Serre fibration.

\subsubsection{Parametric isotopy extension}
\label{sec:param-isot-extens}
As a consequence of the definition in \S\ref{sec:space-cylindrical-lags}, for any \( L \in \mathrm{Lag}(W) \), the map:
\[ \mathrm{Ham}(W) \rightarrow \mathrm{Lag}^{\mathrm{x}}(W) \] given by \( \phi \mapsto \phi(L) \) is a Serre fibration. This can be thought of as a variant of the isotopy extension theorem.

It is an immediate consequence of the parametric isotopy theorem for Legendrians that $\mathrm{Leg}(Y)$ is the same space as the translation-invariant Lagrangians of \( SY \) equipped with the Hamiltonian topology. We refer the reader to \cite[\S2.3]{traynor_helix_links} and \cite[\S2.6]{Geiges} for a proof of the isotopy extension theorem in the case of a single Legendrian; the argument can be repeated nearly verbatim over the parameter space to obtain the parametric version.

\subsubsection{Serre fibration property}
\label{sec:serre-fibr-prop}
In order to prove that the ideal restriction is a Serre fibration we must fill in the lifting diagram:
\begin{equation*}
  \begin{tikzcd}
    {D^n} \arrow[r,"\Phi_{0}"] \arrow[d,hook,swap,"\mathrm{id}\times \set{0}"] & {\mathrm{Lag}^{\mathrm{x}}(W)} \arrow[d]\\
    {D^n \times {[0,1]}} \arrow[r,"\phi"] \arrow[ur,dashed,"\Phi"] & {\mathrm{Leg}(Y)}
  \end{tikzcd}.
\end{equation*}
The existence of such a lift, \( \Phi, \) follows immediately from the parametric isotopy extension theorem for Legendrians, and the contact-at-infinity condition, by cutting off below a fixed symplectization level. In other words, the same proof used for the Serre fibration property for $\mathrm{Ham}(W)\to \mathrm{Cont}(Y)$ from \S\ref{sec:fibr-sequ-sympl} applies here, using parametric isotopy extension.

Often we will use the Serre fibration property when the left hand side is instead: $$([0,1]^{n}\times\set{0})\cup (\bd [0,1]^{n}\times [0,1])\subset [0,1]^{n+1};$$
the fact that the Serre fibration property applies to this domain is well-known.

\subsection{Wrapped Floer cohomology}
\label{sec:wrapped-floer-setup}

The main ideas and definitions in this section are straightforward adaptations of those in \S\ref{sec:floer-cohom-groups} to the open string case.

\subsubsection{Definition of the Floer complex}
\label{sec:floer-data}

Fix a contact-at-infinity Lagrangian $L$ with ideal restriction $\Lambda$. A contact-at-infinity system $\psi_{t}\in \mathrm{Ham}(W)$ is \emph{non-degenerate} provided $\psi_{1}(L)$ is transverse to $L$. This requires that the ideal restriction $\psi_{1}(\Lambda)$ is disjoint from $\Lambda$, and this disjunction plays the role of the discriminant point condition.

Given $J_{t}\in \mathscr{J}$, define $\mathscr{M}(L,\psi_{1}(L),J_{t})$ to be the space of all solutions to:
\begin{equation*}
  \left\{
    \begin{aligned}
      &w:\R\times [0,1]\to W,\\
      &\bd_{s}w+J_{t}(w)\bd_{t}w=0,\\
      &w(s,0)\in \psi_{1}(L)\text{ and }w(s,1)\in L.
    \end{aligned}
  \right.\hspace{1cm}
  \begin{tikzpicture}[baseline={(0,0.1)},scale=0.7]
    \draw (0,0)--node[below]{$\psi_{1}(L)$}+(5,0)(0,1)--node[above]{$L$}+(5,0);
  \end{tikzpicture}
\end{equation*}
Similarly to the closed string case, the choice of system $\psi_{t}$ allows us to change coordinates $u(s,t)=\psi_{1-t}^{-1}(w(s,t))$ so that $u$ solves the normal Floer's equation for the system $\psi_{1-t}^{-1}\psi_{1}$ with both boundaries on $L$; see \S\ref{sec:energy-estim-other} for further discussion.

Let us say that $(\psi_{t},J_{t})$ is \emph{admissible for defining the Floer complex} if $\psi_{1}(L)\cap L$ is transverse and the moduli space $\mathscr{M}(L,\psi_{1}(L),J_{t})$ is cut transversally.

In this case, let $\mathrm{CF}(L;\psi_{t},J_{t})$ be the $\Z/2$-vector space generated by the intersection points of $L$ and $\psi_{1}(L)$, whose differential is given by counting the curves in $\mathscr{M}_{1}/\R$ going from right to left, exactly as in the closed string case.

\subsubsection{Subcomplex generated by contractible chords}
\label{sec:subc-gener-contr-1}

Similarly to \S\ref{sec:subc-gener-contr}, we introduce $\mathrm{CF}_{0}(L;\psi_{t},J_{t})$ be the subcomplex generated by chords which are contractible, i.e., chords $\gamma(t)$ which admit a \emph{capping} $v:[0,1]^{2}\to W$ so that:
\begin{equation*}
  v(1,t)=\gamma(t)\text{ and }v(s,0),v(s,1)\in L\text{ and }v(0,t)=\text{const}.
\end{equation*}
It is clear that $\mathrm{CF}_{0}$ is preserved under the differential and the continuation maps defined in \S\ref{sec:continuation-maps-1}, and we let $\mathrm{HF}_{0}$ denote the cohomology of $\mathrm{CF}_{0}$.

\subsubsection{Continuation maps}
\label{sec:continuation-maps-1}

Let $\psi_{s,t},J_{s,t}$ be a path of Floer data (which is $s$ independent outside of $[s_{0},s_{1}]$). The continuation map moduli space $\mathscr{M}(\psi_{s,t},J_{s,t})$ consists of solutions to:
\begin{equation*}
  \left\{
    \begin{aligned}
      &w:\R\times [0,1]\to W,\\
      &\bd_{s}w+J_{-s,t}(w)\bd_{t}w=0,\\
      &w(s,0)\in \psi_{-s,1}(L)\text{ and }w(s,1)\in L.
    \end{aligned}
  \right.
\end{equation*}
We reverse the sign to $-s$ so as to be consistent with the input being at the positive end.

We say that $\psi_{s,t},J_{s,t}$ is \emph{admissible continuation data} if:
\begin{enumerate}
\item the ideal restriction of $\psi_{s,1}(L)$ is a non-negative Legendrian isotopy,
\item the moduli space $\mathscr{M}$ is cut transversally, and
\item the endpoints of the path are admissible for defining the Floer complex.
\end{enumerate}
The first assumption (i) ensures an a priori energy bound, which implies the requisite compactness results needed to do Floer theory.

In this case, the zero dimensional component $\mathscr{M}_{0}$ consists of finitely many points. Each point determines a map $\mathrm{CF}(L;\psi_{s_{0},t},J_{s_{0},t})\to \mathrm{CF}(L;\psi_{s_{1},t},J_{s_{1},t})$ by sending the right asymptotic to the left asymptotic, and the continuation map $\mathfrak{c}$ is the sum of all these individual contributions.

As usual, consideration of the one-dimensional component $\mathscr{M}_{1}$ implies $\mathfrak{c}$ is a chain map. Parametric moduli spaces imply $\mathfrak{c}$ remains in the same chain homotopy class under deformations $\psi_{s,t}^{\eta},J_{s,t}^{\eta}$ with fixed endpoints $s=s_{0}$ and $s=s_{1}$ and for which the ideal restriction $s\mapsto \psi_{s,1}^{\eta}(L)$ remains non-negative.

Finally, standard gluing arguments imply that the induced maps on $\mathrm{HF}(L;\varphi_{t},J_{t})$ are functorial with respect to concatenation of paths. See, e.g., \cite[\S3.2]{biran-cornea-rigidity-uniruling} and \cite[\S8.j]{seidel_book} for further discussion.

\subsubsection{Energy estimate and other a priori estimates}
\label{sec:energy-estim-other}
Suppose that $\psi_{s,t},J_{s,t}$ is admissible continuation data and the ideal restriction of $\psi_{s,1}(L)$ is a non-negative Legendrian isotopy. The goal in this section is to explain the a priori energy estimates for solutions of $\mathscr{M}(\psi_{s,t},J_{s,t})$, as in \S\ref{sec:energy-estim-ator}. The acylindrical assumption will be used similarly to the atoroidal assumption in the closed-string case.

The secondary goal is to explain how to upgrade the energy estimate to other norms, so as to ensure the necessary compactness results hold; see \cite[\S2.2.4]{cant_sh_barcode} for related discussion.

The energy of $w\in \mathscr{M}(\psi_{s,t},J_{s,t})$ is defined as the integral of $\omega$ over $w$. Introduce the coordinate change: $$u(s,t)=\psi_{-s,1-t}^{-1}(w(s,t)),$$ so that $u$ has the $t=0,1$ boundaries on $L$. Let $Y_{s,t}$ and $X_{s,t}$ be the generators of $\xi_{s,t}:=\psi_{-s,1-t}^{-1}$ with respect to variations in $s$ and $t$. One computes:
\begin{equation*}
  \bd_{s}u-Y_{s,t}=\d\xi_{s,t}\bd_{s}w\hspace{1cm}\bd_{t}u-X_{s,t}=\d\xi_{s,t}\bd_{t}w.
\end{equation*}
Since $\d\xi_{s,t}$ is symplectic, the energy of $w$ satisfies:
\begin{equation*}
  E(w)=\int\omega(\bd_{s}u-Y_{s,t},\bd_{t}u-X_{s,t})\d s\d t,
\end{equation*}
and a standard computation yields:
\begin{equation}\label{eq:open-string-energy-bound}
  E(w)=\int u^{*}\omega+\int_{\R} K_{s,1}(u(s,1))-K_{s,0}(u(s,0))\d s+\int_{\gamma_{-}}H_{-,t}\d t-\int_{\gamma_{+}}H_{+,t}\d t.
\end{equation}
where $H_{s,t},K_{s,t}$ are normalized generators for $X_{s,t}$ and $Y_{s,t}$, and we use the curvature identity:
\begin{equation*}
  \bd_{s}H_{s,t}-\bd_{t}K_{s,t}+\omega(Y_{s,t},X_{s,t})=0.
\end{equation*}
Since $Y_{s,1}$ generates $\xi_{s,1}=\id$ while $Y_{s,0}$ generates $\xi_{s,0}=\psi^{-1}_{-s,1}$, which is a \emph{non-negative path}, we conclude that $K_{1,t}=0$ and $K_{0,t}\ge 0$ outside of a compact set. Consequently, we can bound the second term in \eqref{eq:open-string-energy-bound}. The terms involving $H_{\pm,t}$ can be bounded exactly as in \S\ref{sec:energy-estim-ator} (one uses that the asymptotics are contained in a compact set), while the symplectic area term can be bounded using the assumption that $L$ is symplectically acylindrical. This completes the proof of the a priori energy bound.

We can upgrade this energy bound to an a priori $C^{1}$ bound using bubbling analysis (bearing in mind that the symplectically acylindrical condition implies there are no $J$-holomorphic spheres or disks on $L$). Standard elliptic bootstrapping and the Sobolev embedding theorem ensure that any sequence $w_{n}$ in the moduli space has a subsequence which converges in $C^{\infty}_{\mathrm{loc}}$.

A straightforward variant of the argument in \cite[\S2]{brocic_cant} and \cite[\S2.2.5]{cant_sh_barcode} shows that any sequence $w_{n}\in \mathscr{M}(\psi_{s,t},J_{s,t})$ with bounded energy will remain in a fixed compact set (depending only on the energy bound, the $C^{1}$ bound, and $\psi_{\pm,t}$).

These compactness results are implicitly used when appealing to standard Floer theoretic arguments.

\subsubsection{Wrapped Floer cohomology}
\label{sec:wrapp-floer-cohom}

The open string analogue of $\mathrm{SH}(W)$ is the wrapped Floer cohomology of $L$, denoted $\mathrm{HW}(L)$. It is defined as the colimit of some functor. The colimit of the groups generated by contractible orbits is denoted $\mathrm{HW}_{0}(L)$.

Let $\Delta(L)$ be the small category consisting data $\psi_{t},J_{t}$ admissible data, and whose morphisms $(\psi_{0,t},J_{0,t})\to (\psi_{1,t},J_{1,t})$ consist of homotopy classes of extensions $\psi_{s,t}$ for which the ideal restriction $\psi_{s,1}(L)$ is non-negative. Assigning to each object the Floer cohomology $\mathrm{HF}(L;\psi_{t},J_{t})$, and to each morphism the continuation map (with $J_{s,t}$ chosen arbitrarily and $\psi_{s,t}$ perturbed as to be admissible), defines a functor $\Delta(L)\to \mathrm{Vect}$.

The colimit of this functor is the \emph{wrapped Floer cohomology} $\mathrm{HW}(L)$.

\subsubsection{The Floer cohomology associated to a Legendrian isotopy}
\label{sec:floer-cohom-legendr}

As in the closed string case, the Floer cohomology group $\mathrm{HF}(L;\psi_{t},J_{t})$ depends only the ideal restriction of $\psi_{t}(L)$. To make this precise, fix a Legendrian isotopy $\Lambda_{t}$ so that $\Lambda_{0}=\Lambda$ is the ideal boundary of $L$.

Let $\Delta(L;\Lambda_{t})$ be of the subset of $\Delta(L)$ consisting of elements $(\psi_{t},J_{t})$ so that the ideal restriction of $\psi_{t}(L)$ equals $\Lambda_{t}$.

\begin{prop}
  There is a unique morphism $[\psi_{s,t}]$ in $\Delta(L)$ between two objects in $\Delta(L;\Lambda_{t})$ so that the ideal restriction of $\psi_{s,t}(L)$ is identically equal to $\Lambda_{t}$.
\end{prop}
\begin{proof}
  Let $\psi_{0,t}$ and $\psi_{1,t}$ be two objects. Since $\mathrm{Lag}^{\mathrm{x}}\to \mathrm{Leg}$ is a Serre fibration, there is some Lagrangian $L_{s,t}$ so that $L_{i,t}=\psi_{i,t}(L)$, $L_{s,0}=L$ and the ideal restriction of $L_{s,t}$ is $\Lambda_{t}$; one applies the Serre fibration property to the square $\Lambda_{s,t}=\Lambda_{t}$ which is already lifted at $s=0,1$ and $t=0$.

  Next one applies the parametric isotopy extension theorem of \S\ref{sec:param-isot-extens} to conclude some extension $\psi_{s,t}$ so that $\psi_{s,t}(L)=L_{s,t}$. Since the ideal restriction of $\psi_{s,t}(L)$ is $s$-independent, it defines a continuation map which is invertible.

  Thus we have proved there is some morphism. Now we will prove it is unique. If $\psi_{s,t}^{0},\psi_{s,t}^{1}$ both have ideal restriction $\Lambda_{t}$, then we can first find $L_{s,t}^{\eta}$ so:
  \begin{enumerate}
  \item $L_{s,t}^{i}=\psi_{s,t}^{i}(L)$ for $i=0,1$,
  \item $L_{i,t}^{\eta}=\psi_{i,t}(L)$ for $i=0,1$, (bear in mind $\psi^{i}_{s,t}$ have same endpoints $s=0,1$),
  \item $L_{s,0}^{\eta}=L$,
  \item the ideal restriction of $L_{s,t}^{\eta}$ is $\Lambda_{t}$.
  \end{enumerate}
  One applies the Serre fibration property to the cube $\Lambda_{s,t}^{\eta}=\Lambda_{t}$, which has a lift on all but one of its faces. Then the parametric isotopy extension theorem gives an extension $\psi_{s,t}^{\eta}$, which proves the morphism is unique.
\end{proof}

We endow $\Delta(L;\Lambda_{t})$ with these canonically defined morphisms, so that there is a unique morphism between any two objects. As in the closed string case, this enables us to define:
\begin{equation*}
  \mathrm{HF}(L;\Lambda_{t})=\mathop\mathrm{colim}_{\Delta(L;\Lambda_{t})}\mathrm{HF}(L;\varphi_{t},J_{t}).
\end{equation*}

One can consider the smaller category $\Delta'(L)$ whose objects are isotopies $\Lambda_{t}$ so that $\Lambda_{0}=\Lambda$ is the ideal boundary of $L$, and whose morphisms are paths $\Lambda_{s,t}$ so that $\Lambda_{s,1}$ is a non-negative path. For the same reasons as the closed string case \S\ref{sec:floer-cohom-assoc}, the assignment $\Lambda_{t}\in \Delta'(L)\mapsto \mathrm{HF}(L;\Lambda_{t})$ is functorial, and the colimit of this simpler functor recovers $\mathrm{HW}(L)$.

We also remark that $\mathrm{HF}(L,\Lambda_{t})$ depends only on $(\Lambda_{1},[\Lambda_{t}])$ (and $L$) up to isomorphism, although it seems difficult to canonically obtain an isomorphism; there seems to be an ambiguity due to $\pi_{2}(\mathrm{Leg})$.

Similar arguments apply to defining $\mathrm{HF}_{0}(L;\Lambda_{t})$, the cohomology generated by the contractible orbits also depends only on the ideal restriction.

\subsubsection{Dominating sequences}
\label{sec:dominating-sequences}

Similar to the closed string case in \S\ref{sec:domin-sequ}, call a functor \( \mathfrak{N}:\mathbb{N} \to \Delta(L) \), \emph{dominating} if there is a $C^{\infty}$ open neighborhood $U$ of the constant paths of Legendrians and representatives $\psi_{i,s,t}$ for $\mathfrak{N}(i)\to \mathfrak{N}(i+1)$ so that $\psi_{i,s,1}(\Lambda_{s})$ is positive for all paths $\Lambda_{s}$ contained in $U$.

An analogous proof to the closed string case shows that dominating functors are filtering in \( \Delta(L) \) and are thus cofinal by Lemma \ref{lemma:filter-cofinal}.

\subsection{PSS morphisms and the product structure on wrapped Floer cohomology}
\label{sec:pss-morph-prod}

The result we need from this section is that the colimit map:
\begin{equation*}
  \mathrm{HF}(L;R_{-\epsilon t}^{\alpha})\to \mathrm{HW}(L)
\end{equation*}
is not surjective, unless $\mathrm{HW}(L)$ is zero. This non-surjectivity is used in \S\ref{sec:natur-transf-assoc} to complete the proof of Theorem \ref{theorem:main-relative}.

As in \S\ref{sec:when-unit-born}, the idea is to consider the unit element, and to argue that a nilpotent element in $\mathrm{HW}(L)$ can never equal the unit. This part of the argument is sensitive to the existence of holomorphic disks, and we use the acylindrical assumption (at least, the weakly exact assumption; see \S\ref{sec:pair-pants-product}) in a crucial way.

\subsubsection{Pair-of-pants product in the open string case}
\label{sec:pair-pants-product}
Analogously to the closed string discussion in \S\ref{sec:PSS-and-the-unit-sympl-cohom}, we have the following result characterizing the relevant algebraic structure of wrapped Floer cohomology $\mathrm{HW}_{0}(L)$ (generated by the contractible chords).

The theorem holds for Lagrangians $L$ which are \emph{weakly exact}, i.e., so that $\omega$ vanishes on $\pi_{2}(W,L)$. This condition is the open-string analogue to the aspherical condition, and is a weaker assumption than being acylindrical.

\begin{theorem}[{\cite[Theorems 6.13, 6.14]{ritter_TQFT}}]\label{thm:hw-ring-str}
  For \( L \subset W \) a weakly-exact and contact-at-infinity Lagrangian, the wrapped Floer cohomology \( \mathrm{HW}_{0}(L) \), equipped with the ``triangle product'' is a ring with unit. Moreover there is a (unital) ring homomorphism:
  \begin{equation}
    H^{*}(L)\simeq \mathrm{HF}_{0}(L;R_{\epsilon t}^{\alpha})=\mathrm{HF}(L;R_{\epsilon t}^{\alpha}) \to \mathrm{HW}_{0}(L),
  \end{equation}
  where \( H^*(L) \) is equipped with the cup product, provided $\epsilon>0$ is sufficiently small.
\end{theorem}

The reason we restrict to the case of contractible chords is in order to have a priori symplectic area estimates for the ``pairs-of-pants'' (or ``triangles'') which appear when defining the product. One can obtain the required a priori energy estimates for the product between non-contractible chords at the expense of working over the Novikov field, or assuming the Lagrangian is sufficiently exact (one should require that $\omega$ vanishes on every pair-of-pants, i.e., sphere with three disks removed, with boundary on $L$). See \S\ref{sec:subc-gener-contr} for related discussion in the closed string case.

\subsubsection{PSS morphisms and the birth of the unit}
\label{sec:pss-morphisms-birth}

As in \S\ref{sec:pss-comparison}, the continuation map $\mathrm{HF}(L;R_{-\epsilon t}^{\alpha})\to \mathrm{HF}(L;R_{\epsilon t}^{\alpha})$ can be factorized as:
\begin{equation}\label{eq:lag-factorization}
  \begin{tikzcd}
    {\mathrm{HF}(L;R_{-\epsilon t}^{\alpha})}\arrow[d,"{\mathrm{PSS}}"]\arrow[r,"{\mathfrak{c}_{\mathrm{F}}}"] &{\mathrm{HF}(L;R_{+\epsilon t}^{\alpha})}\arrow[from=2-2,"{\mathrm{PSS}}"]\\
    {\mathrm{HM}(L;f_{-})}\arrow[r,"{\mathfrak{c}_{\mathrm{M}}}"] &{\mathrm{HM}(L;f_{+})},
  \end{tikzcd}
\end{equation}
where the lower map $\mathfrak{c}_{\mathrm{M}}$ is the Morse continuation map defined by counting continuation lines (the Morse theory part happens entirely in $L$).

\begin{figure}[H]
  \centering
  \begin{tikzpicture}
    \draw (0,1) to[out=-10,in=190] (4,1) (0,0) to[out=10,in=170] (4,0);
    \node at (0,0.5) [left] {$\mathrm{HF}(L;R_{+\epsilon t}^{\alpha})$};
    \node at (4,0.5) [right] {$\mathrm{HF}(L;R_{-\epsilon t}^{\alpha})$};
    \begin{scope}[shift={(0,-2)}]
    \node at (0,0.5) [left] {$\mathrm{HF}(L;R_{+\epsilon t}^{\alpha})$};
    \node at (4,0.5) [right] {$\mathrm{HF}(L;R_{-\epsilon t}^{\alpha})$};
      \draw (0,1) to[out=-10,in=10,looseness=3] coordinate(X) (0,0)  (4,1) to[out=190,in=170,looseness=3]coordinate(Y) (4,0);
      \draw[every node/.style={circle,fill=black,inner sep=1pt}] (X) to[out=0,in=180] +(.7,0.2)node{} to[out=0,in=180] +(1.4,-0.2)node{} to[out=0,in=180](Y);
    \end{scope}
  \end{tikzpicture}
  \caption{Deforming the continuation map so as to factorize it; compare Figures \ref{fig:PSS} and \ref{fig:PSS-more-detail}.}
  \label{fig:deforming-pss-so-as-to-factorize}
\end{figure}
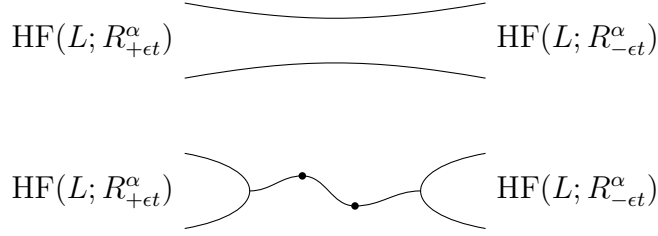

Exactly as in \S\ref{sec:pss-comparison} it is important that $-\epsilon<0<\epsilon$, so that the continuation strips satisfy the required non-negativity for the a priori energy bounds in \S\ref{sec:energy-estim-other}. We leave the details of the factorization \eqref{eq:lag-factorization} to the reader.

Then, by degree reasons, since the ideal boundary $\Lambda$ is presumed to be non-empty, the Morse continuation map:
\begin{equation*}
  \mathfrak{c}_{M}:\mathrm{HM}(L;f_{-})\to \mathrm{HM}(L;f_{+})
\end{equation*}
hits only elements of positive degree (one can take $f_{-}$ to have no local minima); in particular, $\mathfrak{c}_{M}$ hits only \emph{nilpotent} elements with respect to the cup product. By Theorem \ref{thm:hw-ring-str}, the image of:
\begin{equation*}
  \mathrm{PSS}\circ \mathfrak{c}_{M}:\mathrm{HM}(L;f_{-})\to \mathrm{HF}(L;R_{+\epsilon t}^{\alpha})
\end{equation*}
hits only nilpotent elements with respect to the triangle product. Here $\epsilon>0$ should be taken smaller than the minimal action of a Reeb chord, in order for $\mathrm{HF}(L;R_{+\epsilon t}^{\alpha})$ to be a unital ring.

Finally, by the factorization \eqref{eq:lag-factorization}, it follows that $\mathrm{HF}(L;R_{-\epsilon t}^{\alpha})\to \mathrm{HW}(L)$ also hits only nilpotent elements, and hence cannot hit a non-zero unit, as desired.

If the unit is zero, then $\mathrm{HW}_{0}(L)=0$, and hence $\mathrm{HW}(L)$ is supported only on non-contractible chords; because $\mathrm{HF}(L;R^{\alpha}_{-\epsilon t})$ is supported only on contractible chords, we conclude that:
\begin{lemma}\label{lemma:non-surjective-lemma-lag}
  If $\mathrm{HW}(L)\ne 0$, then $\mathrm{HF}(L;R^{\alpha}_{-\epsilon t})\to \mathrm{HW}(L)$ is not surjective.\hfill$\square$
\end{lemma}
Compare with Lemma \ref{lemma:non-surjective-lemma} from the closed string case.

\subsection{Naturality transformation associated to a Lagrangian loop}
\label{sec:natur-transf-assoc}

In this section we explain the open-string analogue of the naturality transformations of \S\ref{sec:natur-transf}, and then describe the open string analogue of the twisting trick from \S\ref{sec:the-twist-trick}.

\subsubsection{Naturality transformations}
\label{sec:natur-transf-1}

Let $\varphi_{t}$ be a Hamiltonian system so that $\varphi_{t}(L)$ forms a loop of Lagrangians. We do not require that $\varphi_{1}=\id$, but rather only that $\varphi_{1}(L)=L$.

For appropriate $J'_{t}$ depending on $J_{t}$ and $\varphi_{1}$, $\mathrm{CF}(L;\psi_{t},J_{t})$ and $\mathrm{CF}(L;\varphi_{t}\psi_{t},J_{t})$ are isomorphic as chain complexes; the isomorphism sends $x$ to $\varphi_{1}(x)$. Taking homology induces the naturality transformation $$\mathfrak{n}:\mathrm{HF}(L;\psi_{t})\to \mathrm{HF}(L;\varphi_{t}\psi_{t}).$$

Indeed, if $w$ lies in $\mathscr{M}(L,\psi_{1}(L),J_{t})$ then $\varphi_{1}(w)$ lies in $\mathscr{M}(L,\varphi_{1}(\psi_{1}(L)),J_{t}')$. This correspondence between Floer cylinders for $\psi_{t}$ and for $\varphi_{t}\psi_{t}$ is used to show that the differentials on the two Floer complexes are identified under the naturality transformation.

Similarly to the closed string case, the naturality transformation commutes with continuation maps. This is proved by showing that $w$ solves the continuation equation for $\psi_{s,t}$ implies $\varphi_{1}(w)$ solves the continuation map equation for $\varphi_{t}\psi_{s,t}$, for appropriate choices of complex structure.

\subsubsection{The twisting trick}
\label{sec:twisting-trick}

One first concludes $\mathrm{HF}(\psi_{t})\to \mathrm{HW}(L)$ is surjective, for any admissible system $\psi_{t}$, exactly as in \S\ref{sec:the-twist-trick}.

On the other hand, we have shown in \S\ref{sec:pss-morphisms-birth} that $\mathrm{HF}(\psi_{t})\to \mathrm{HW}(L)$ is not surjective for $\psi_{t}=R_{-\epsilon t}^{\alpha}$, unless $\mathrm{HW}(L)$ is zero. Thus we conclude that $\mathrm{HW}(L)$ is zero, completing the proof of Theorem \ref{theorem:main-relative}.

\subsection{Non-isotopic Lagrangian fillings and the flux group}
\label{sec:lagrangian-flux}

This section is concerned with the Lagrangian version of \S\ref{sec:exot-sympl-flux}. Let $L_{t}$ be a smooth loop of contact-at-infinity Lagrangians based at an acylindrical Lagrangian $L$, and let $\Lambda_{t}$ be the ideal restriction of $L_{t}$. The goal is to prove the existence of a deformation $L_{t}'$ which is an \emph{exact} loop based at $L$ with the same ideal restriction $\Lambda_{t}$. The proof of this assertion uses the acylindrical assumption, and a relative version of flux; see \cite{ono-lagrangian-flux,solomon-lagrangian-flux,shelukhin-tonkonog-vianna,entov-ganor-membrez} for further discussion of the Lagrangian flux.

First of all, pick a contact-at-infinity system $\psi_{t}$ so that the ideal restriction of $\psi_{t}(L)$ agrees with the ideal restriction of $L_{t}$; this can be achieved by the isotopy extension theorem for Legendrians; see \S\ref{sec:param-isot-extens}. Then $\psi_{t}^{-1}(L_{t})$ is a compactly supported deformation of $L$. Let $\mathfrak{L}_{t}=\psi_{t}^{-1}(L_{t})$ with $\mathfrak{L}_{0}=L$.

It is sufficient to prove that there is a compactly supported Hamiltonian isotopy $\rho_{t}$ so that $\rho_{1}(L)=\mathfrak{L}_{1}$; then $\psi_{t}\rho_{t}(L)$ is an exact loop whose ideal restriction agrees with $L_{t}$.

Since $\mathfrak{L}_{t}$ are compactly supported deformations of $L$, there exists a compactly supported parametrization $j_{t}:L\to W$ so that $j_{t}(L)=\mathfrak{L}_{t}$.

Define the flux one-forms on $L$ by: $$F(j_{t};\tau)_{x}:=\int_{0}^{\tau}\omega(j'_{t}(x),\d j_{t,x}(-))\d t;$$
Cartan's magic formula implies this is a closed one-form defined on $L$.

Moreover, if $\Gamma$ is any oriented loop in $L$, the integral of $F(j_{t};\tau)$ over $\Gamma$ equals the integral of $\pm\omega$ over the trace $\Gamma\times [0,\tau]$ under the parametrization $j_{t}$. In particular, $[F(j_{t};1)]=0$ since $j_{1}(L)=L$ and $L$ is acylindrical; here $[-]$ denotes the cohomology class in $\mathrm{H}^{1}_{\mathrm{dR}}(L)$. To see why this flux vanishes, observe that, if $u(s,t)$ parametrizes the trace $\Gamma\times [0,1]$, then $\psi_{t}(u(s,t))$ parametrizes a cylinder with both boundaries on $L$, and therefore has zero symplectic area. As in \S\ref{sec:vanish-flux-exist}, the symplectic area of $\psi_{t}(u(s,t))$ equals the symplectic area of $u(s,t)$ since $\psi_{t}$ is Hamiltonian.

This discussion implies that the following lemma is sufficient to complete the proof of Theorem \ref{theorem:lagrangian-exotic}.
\begin{lemma}\label{lemma:lagrangian-flux-solomon}
  If $j_{t}:L\to W$ is a compactly supported Lagrangian isotopy satisfying $[F(j_{t};1)]=0$, and every compactly supported closed one-form on $L$ extends to one on $W$, then there is a compactly supported Hamiltonian isotopy $\rho_{t}$ so that $\rho_{1}(L)=j_{1}(L)$.
\end{lemma}
\begin{proof}
  One shows that one can make $[F(j_{t};\tau)]=0$ for all $\tau$ without changing $j_{1}(L)$ and $j_{0}(L)=L$ by deforming $j_{t}$ to another compactly supported Lagrangian isotopy. Differentiating the flux with respect to $\tau$ proves the deformation is an exact isotopy; see, e.g., \cite[Corollary 6.4]{solomon-lagrangian-flux}. Then the isotopy extension theorem for exact Lagrangian isotopies applies.

  To make $[F(j_{t};\tau)]=0$ for all $\tau$ one applies \cite[Lemma 6.7]{solomon-lagrangian-flux}. The idea is to define a right inverse $\mathfrak{s}:\mathrm{H}_{\mathrm{dR}}^{1}(L)\to \mathrm{H}_{\mathrm{dR}}^{1}(W)$. One only cares about the subspace spanned by compactly supported forms, and defines $s$ using some choice of basis.

  One finds a compactly supported symplectic isotopy $\psi_{s,t}:[0,1]^{2}\to \mathrm{Symp}_{c}(W)$ so that the symplectic flux along the line $\psi_{s,\tau}$ is $-\mathfrak{s}([F(j_{t};\tau)])$ as $s$ ranges from $0$ to $1$, and $\psi_{0,\tau}=\id$. Then $\psi_{1,t}j_{t}$ is the desired exact isotopy; see \cite{solomon-lagrangian-flux} for the details. This completes the proof.
\end{proof}

\subsubsection{Removing the assumption on the first de Rham cohomology}
\label{sec:remov-assumpt-first}

It seems to be an interesting question whether the surjectivity of the first deRham cohomology $\mathrm{H}^{1}_{\mathrm{dR}}(W)\to \mathrm{H}^{1}_{\mathrm{dR}}(L)$ can be relaxed in the statement of Lemma \ref{lemma:lagrangian-flux-solomon}.

Interestingly enough, not every cohomology class in $\mathrm{H}^{1}_{\mathrm{dR}}(L)$ represented by a compactly supported form can be realized as the flux of a Lagrangian isotopy; see \cite[\S6.2]{shelukhin-tonkonog-vianna} for the case $L=T(r)=\bd D(r_{1})\times \dots\times \bd D(r_{n})\subset \C^{n}$; they determine the exact subset of $\mathrm{H}^{1}_{\mathrm{dR}}(L)$ which can be realized as the flux of a Lagrangian isotopy (this set is called the \emph{shape} of $L$ relative $W$). The characterization of when two tori $T(r),T(r')$ are Hamiltonian isotopic is given in \cite{chekanov1996}, where it is also shown that there are two isotopic monotone tori $T,T'$ in $\R^{2n}$, with the same monotonicity constant, which are not Hamiltonian isotopic. Since any isotopy between $T,T'$ has zero flux (since they have the same monotonicity constant), one sees that Lemma \ref{lemma:lagrangian-flux-solomon} is not true in general;\footnote{Thanks to D.~Rathel-Fournier for pointing this out.} see also \cite[Remark 1.8]{theret-camel}.

\bibliographystyle{alpha}
\bibliography{citations}
\end{document}